\documentclass{siamltex1213}
\usepackage{amsmath,amsfonts,amssymb} 
\usepackage{algorithm,algpseudocode}
\usepackage{xstring,blkarray}
\usepackage[final]{changes}
\usepackage[section]{placeins}
\algblockdefx[mfunction]{MFunction}{EndMFunction}[3]{\textbf{function} #1$=$\texttt{#2}(#3)}{\textbf{end function}}
\algblockdefx[arguments]{Arguments}{EndArguments}{\textbf{argument definitions}}{\textbf{end definitions}}
\algblockdefx[mdo]{MDo}{EndMDo}[1]{\textbf{do} #1}{\textbf{end do}}
\algblockdefx[mdo]{}{MWhile}[1]{\textbf{while} #1}

\newcommand{\wide}[1]{\mbox{\StrSubstitute{#1}{ }{\quad}}} 

\newcommand{\norm}[1]{\lVert #1 \rVert}

\begin{document}

\title{Randomized QR with Column Pivoting\thanks{University of California, Berkeley. This research was supported in part by NSF Award CCF-1319312.}}

\author{Jed A. Duersch \and Ming Gu}

\maketitle

\begin{abstract}
The dominant contribution to communication complexity in factorizing a matrix using QR with column pivoting is due to column-norm updates that are required to process pivot decisions.
We use randomized sampling to approximate this process which dramatically reduces communication in column selection.
We also introduce a sample update formula to reduce the cost of sampling trailing matrices.
Using our column selection mechanism we observe results that are comparable in quality to those obtained from the QRCP algorithm, but with performance near unpivoted QR.
We also demonstrate strong parallel scalability on shared memory multiple core systems using an implementation in Fortran with OpenMP.

This work immediately extends to produce low-rank truncated approximations of large matrices.
We propose a truncated QR factorization with column pivoting that avoids trailing matrix updates which are used in current implementations of level-3 BLAS QR and QRCP.
Provided the truncation rank is small, avoiding trailing matrix updates reduces approximation time by nearly half.
By using these techniques and employing a variation on Stewart's QLP algorithm, we develop an approximate truncated SVD that runs nearly as fast as truncated QR.

\medskip
    
\newline blocked algorithm, low-rank approximation, truncated SVD
\end{abstract}

\section{Introduction}

We explore a variation of QR with Column Pivoting (QRCP) using randomized sampling (RQRCP) to process blocks of pivots.
Magnitudes of trailing column norms are detected using Gaussian random compression matrices to produce smaller sample matrices.
We analyze the probability distributions of sample column norms to justify the internal updating computation used to select blocks of column pivots.
P.G. Martinsson~\cite{Martinsson2015} independently developed a very similar approach in parallel to this research.
The primary difference is our introduction of a sample update formula which reduces matrix-multiplication complexity required to process the full matrix factorization by one third of what would be required by re-sampling after processing each block of column pivots.

We also extend this method of factorization to produce truncated low-rank approximations.
We propose an implementation that avoids the trailing update computation on the full matrix.
This further reduces time spent in matrix multiplication to nearly half of what would be required by a truncated version with trailing update that also employs one of our sample update formulas.
Furthermore, the Truncated Randomized QR with Column Pivoting algorithm (TRQRCP) immediately extends to approximate the truncated SVD using a variation on Stewart's QLP algorithm~\cite{journals/siamsc/Stewart99}.

We are able to achieve matrix factorizations of similar quality to standard QRCP while retaining communication complexity of unpivoted QR.
Algorithms have been implemented and tested in Fortran with OpenMP on shared-memory 24-core systems.
Our performance experiments compare these algorithms against LAPACK subroutines linked with the Intel Math Kernel Library and verify computation time nearly as short as unpivoted QR (\texttt{dgeqrf}) and substantially shorter than QRCP (\texttt{dgeqp3}).

We also examine performance and quality of low-rank truncated approximations.
The truncated algorithm TRQRCP operates in the same time domain as our implementation of truncated QR, but it yields approximation error as small as truncated QRCP.
Similarly, the approximate truncated singular value decomposition proposed TUXV yields error nearly as small as the truncated SVD.

By using randomized sampling to process pivoting decisions with a much smaller matrix, we eliminate the leading-order term of communication complexity that causes QRCP to perform much worse than blocked QR on large matrices.
These algorithms have the potential to dramatically reduce factorization times in a wide variety of applications in numerical linear algebra.
Problems that have been too large to process with QRCP-dependent subroutines will now become feasible.
This has the potential to increase the usefulness of computational modeling in a wide variety of fields of research in science and engineering.

\section{QR with column pivoting}

The QR decomposition is one of the most well known and useful tools in numerical linear algebra.
An input matrix $A$ is expressed as the product of an orthogonal matrix $Q$ and a right-triangular factor $R$, $A = QR$.
QR is particularly stable in that the decomposition always exists regardless of the conditioning of the input matrix.
Furthermore, it has finely tuned implementations that operate at the BLAS-3 level of performance.
Unfortunately, the standard QR algorithm is not suitable for purposes requiring rank detection or low-rank approximations.
These objectives require a column permutation scheme to process more representative columns earlier in the decomposition~\cite{C1987}.

The QRCP algorithm is a standard solution that is usually adequate for such purposes with a few rare exceptions such as the Kahan matrix~\cite{golub13}.
A permutation matrix $P$ is introduced to rearrange columns into a more beneficial ordering which is then decomposed as before,
\[
	AP = QR
	\mbox{.}
\]

QRCP is an effective alternative to the much more costly Singular Value Decomposition (SVD) and it has a number of applications including least-squares approximations which are discussed in detail by Chan and Hansen~\cite{CH1992}.
Furthermore, some applications for which unpivoted QR is usually sufficient occasionally encounter badly-behaved matrices.
If QRCP can be implemented at a level of performance similar to QR, stability safeguards can be included to safely handle problematic cases without burdening performance.

\subsection{QRCP performance}

Early implementations of QR and QRCP relied on level-2 BLAS kernels and therefore gave very similar performance results until reflector blocking was employed in QR~\cite{BischofVanLoan1987}\cite{SchreiberVanLoan1989}.
Instead of updating the entire matrix with each Householder reflection, transformations are collected into blocks which are applied using a level-3 BLAS matrix multiply.
This improves performance by reducing slow communication which is communication between slow and fast levels of memory.
If each pivot decision is interleaved with a Householder reflection then all trailing matrix data residing in slow memory must pass through fast memory at least once per iteration.
This gives a leading order of slow communication complexity $b m n$ to process $b$ pivots of an $m \times n$ matrix.
If reflectors are blocked and applied all at once using matrix multiply then the complexity is reduced to $O(b m n / M^{3/2})$ where $M$ is the effective size of fast memory.

The primary obstacle to high-performance QRCP is the additional communication that is required to make column pivoting decisions.
In order to understand how our algorithms improve performance, we first review the reasons why additional communication could not be avoided with previous approaches.
An outline of level-2 BLAS QRCP is provided in Algorithm~\ref{alg:qrcp}.
In the algorithms and derivations that follow, the state of an array at the end of iteration $j$ is denoted by superscript $(j)$.
Superscript $(0)$ refers to the initial state of an array upon entry to the first iteration of the main loop.

\begin{algorithm}
\caption{QRCP with level-2 BLAS Householder reflections.}
\label{alg:qrcp}
\begin{algorithmic}[1]
\Require
\Statex $A$ is $m \times n$.
\Ensure
\Statex $Q$ is $m \times m$ orthogonal matrix.
\Statex $R$ is $m \times n$ right triangular matrix, diagonals in non-increasing magnitude order.
\Statex $P$ is $n \times n$ permutation matrix such that $AP=QR$.
\MFunction {$[Q,R,P]$}{qrcp}{$A$}
\State Compute initial column 2-norms which will become trailing column norms.
\MDo{$j=1,2,\ldots k$ where $k=\min(m,n)$}
   \State Find index $p_j$ of the column with maximum trailing 2-norm.
   \State Apply permutation $S_j$ swapping column $j$ with $p_j$.
   \State Form Householder reflection $H_j = I - y_j \tau_j y_j^T$ from new column $j$.
   \State Apply reflection $A^{(j)} = H_j A^{(j-1)} S_j$.
   \State Update trailing column norms by removing the contribution of row $j$.
\EndMDo
\State $Q = H_1 H_2 \ldots H_k$ is the product of all reflections.
\State $R = A^{(k)}$.
\State $P = S_1 S_2 \ldots S_k$ is the aggregate column permutation.
\EndMFunction
\end{algorithmic}
\end{algorithm}

QRCP can be understood as a greedy procedure intended to maximize the magnitude of the determinant of the upper left block at every iteration.
At the end of iteration $j$ we can represent the matrix $A$ as a partial factorization using the permutation $P_j = S_1 \ldots S_j$ , which is the composition of column swaps so far, and the composition of Householder reflections $Q_j = H_1 \ldots H_j$ to obtain 
\[
	A P_j = Q_j
		\left[ \begin{array}{cc}
			 R^{(j)}_{11} & R^{(j)}_{12} \\
			 0      & \hat{A}^{(j)} \\
		\end{array} \right]
	\mbox{.}
\]
On the next iteration the 2-norm of the selected column within the trailing matrix $\hat{A}^{(j)}$ will become the magnitude of the next diagonal element in $R^{(j+1)}_{11}$.
The new determinant magnitude is $| \det R^{(j+1)}_{11} | = | \det R^{(j)}_{11} | \norm{\hat{A}^{(j)}(:,p_{j+1})}_2$.
This scheme selects the pivot that multiplies the previous determinant by the largest factor at each iteration.
Note that true determinant maximization would require exchanging prior columns and adjusting the factorization accordingly~\cite{GE1996}.

In order to produce a correct pivot decision on iteration $(j+1)$, trailing column norms must be updated to remove the contribution of row $j$ which depends on the Householder transformation $H_j = I - y_j \tau_j y_j^T$ where $y_j$ denotes the reflection vector and $\tau_j = 2 / y_j^T y_j$ is the corresponding reflection coefficient.
The full update requires two level-2 BLAS operations on the trailing matrix per iteration.
The first  operation computes scaled inner products $w_j^T = \tau_j y_j^T A^{(j-1)} S_j $ and the second operation modifies the trailing matrix with the rank 1 update $A^{(j)} = A^{(j-1)} S_j - y_j w_j^T$.
These operations are the performance bottleneck of QRCP.

\subsection{Attempts to achieve BLAS level-3 performance}

Quintana-Ort\'{\i} \textit{et al.\@} were able to halve level-2 BLAS operations with the insight that the trailing norm update does not require forming the full rank 1 update on each iteration~\cite{Quintana-Orti:1998:BVQ:300151.300158}.
Instead reflections can be gathered into blocks as is done for QR.
This is shown in Algorithm~\ref{alg:qrcp3}.

At the end of iteration $j$, the algorithm will have collected a block of reflectors thus far $Y_j$.
Reflector $y_i$ for $i \leq j$ appears in column $i$.
This forms a block reflection $Q_j = I - Y_j T_j Y_j^T$ where $T_j$ is the upper triangular $j \times j$ connection matrix that can be solved from $Y_j$ to give orthogonal $Q_j$.
The algorithm must also collect each corresponding scaled inner product $w_i^T$ which appears in row $i$ of $W_j^T$.
Effectively, $W_j^T = T_j^T Y_j^T A P_j$.
This provides enough information to update row $j$ alone and adjust trailing column norms to prepare for the next pivot selection.

\[
    A^{(j)}(j, :) = A^{(j-1)}(j, :) S_j - Y_j(j,:) W_j^T
\]

Note however that this construction will complicate reflector formation.
As before, the next pivot index $p_{j+1}$ is selected and swapped into column $j+1$.
Call this new column $a_{j+1}$.
From row $j+1$ down, elements of $a_{j+1}$ have not been updated with the current block of reflectors.
Before the new reflector $y_{j+1}$ can be formed, prior transformations must be applied to these rows using $\hat{a}_{j+1} = a_{j+1} - Y_j W_j^T(:, p_{j+1})$.
An additional step is also required to form reflector inner products in order to account for reflections that have not been applied to the trailing matrix.
The adjusted formula for these inner products is
\[
    w_{j+1}^T = \tau_{j+1} \left( y_{j+1}^T A^{(j)} - (y_{j+1}^T Y_j) W_j^T \right) S_{j+1}.
\]
Finally, the reflector and inner product blocks can be updated:
\[
    Y_{j+1} =
    \begin{bmatrix}
        Y_j & y_{j+1} \\
    \end{bmatrix}
    \quad\mbox{and}\quad
    W_{j+1}^T = 
    \begin{bmatrix}
        W_j^T S_{j+1} \\
        w_{j+1}^T \\
    \end{bmatrix}
    \mbox{.}
\]

Unfortunately, the remaining level-2 BLAS operations $y_{j+1}^T A^{(j)}$ and $y_{j+1}^T Y_j$ in the inner product computation still dominate slow communication complexity for large matrices.
The entire trailing matrix must still pass from slow to fast memory once per iteration.
This is why even heavily optimized implementations of blocked QRCP still run substantially slower than blocked $QR$ on both sequential and parallel architectures.

\begin{algorithm}
\caption{QRCP with level-3 BLAS reflection blocking.}
\label{alg:qrcp3}
\begin{algorithmic}[1]
\Require
\Statex $A$ is $m \times n$.
\Ensure
\Statex $Q$ is $m \times m$ orthogonal matrix.
\Statex $R$ is $m \times n$ right triangular matrix, diagonals in non-increasing magnitude order.
\Statex $P$ is $n \times n$ permutation matrix such that $AP=QR$.
\MFunction {$[Q,R,P]$}{qrcp}{$A$}
\State Compute initial column 2-norms which will become trailing column norms.
\MDo{$i=0,b,2b \ldots $ where $b$ is block size.}
\MDo{$j=i+1,i+2,\ldots \min(i+b,k)$ where $k=\min(m,n)$}
   \State Find index $p_j$ of the column with maximum trailing 2-norm.
   \State Apply permutation $S_j$ swapping column $j$ with $p_j$.
   \State \textbf{Update column $j$ with prior reflections in this block.}
   \State Form reflector $y_j$ and $\tau_j$ from new column $j$.
   \State \textbf{Compute adjusted reflector inner products $w_j^T$.}
   \State \textbf{Update row $j$ with all reflections in this block.}
   \State Update trailing column norms by removing the contribution of row $j$.
\EndMDo
\State \textbf{Apply block reflection to trailing matrix.}
\EndMDo
\State $Q = I - Y_k T_k Y_k^T$ where $T_k$ can be recovered from $Y_k$ and $\tau_1, \ldots, \tau_k$.
\State $R = A^{(k)}$.
\State $P = S_1 S_2 \ldots S_k$ is the aggregate column permutation.
\EndMFunction
\end{algorithmic}
\end{algorithm}

\subsection{Communication Avoiding Rank-Revealing QR}

Several mechanisms have been put forward to avoid repeating full passes over the trailing matrix on each iteration.
Bischof proposed pivoting restricted to local blocks~\cite{B1991} and Demmel \textit{et al.\@} propose a procedure called Communication Avoiding Rank-Revealing QR (CARRQR)~\cite{DGGX2013,journals/siamsc/DemmelGHL12}.
CARRQR proceeds by partitioning the trailing matrix into $\mathcal{P}$ subsets of columns that are processed independently and possibly simultaneously.
From within each column subset, $b$ candidate pivots are selected using QRCP.
Adjacent subsets of candidates are then combined into $\frac{1}{2}\mathcal{P}$ subsets of $2b$ candidates which are refined again using QRCP into $b$ new candidates each.
This procedure continues until only one subset of $b$ candidates remains.
The trailing matrix is then updated as before with blocked reflections.

We now examine several practical constraints in implementing CARRQR.
First, both the reflectors $Y$, inner products $W^T$, and leading rows of $R$ must be stored separately from the original matrix for each independently processed subset of columns.
Furthermore, one must employ a version of QRCP that avoids the trailing update.
This is because the final reflectors are unknown until the last selection stage.
Any intermediate changes to the original columns would have to be undone before the final transformations can be correctly processed.
In contrast, QRCP can be written to convert columns into reflectors in place on the strictly lower triangle portion of the matrix array.
Likewise, leading rows of $R$ can be stored on the upper triangle.

Depending on the initial column partition, CARRQR performs between 1 and 2 times as many inner products as QRCP per block iteration.
Note that as the reflector index $j$ increases the total number of inner products of the form $y_{j+1}^T y_{j+1}$, $y_{j+1}^T Y_j$, and $y_{j+1}^T A^{(j)}$ remains constant.
Therefore, if the $i$\textsuperscript{th} column subset contains $n_i$ columns, $b n_i$ inner products will be required to produce $b$ candidates.
Given $n_1 + n_2 + \cdots + n_\mathcal{P} = n$ on the first stage of refinement, summing over all column subsets gives $b n$ inner products to produce $\mathcal{P}$ sets of $b$ candidates.
QRCP requires the same complexity to produce $b$ final pivots.
Assuming that the number of candidates is at least halved for each subsequent stage of refinement in CARRQR, one can easily show that no more than $2 b n$ inner products will be computed in total.

Despite increased computational complexity, CARRQR is intended to benefit from better memory utilization and better parallel scalability.
If each column subset is thin enough to fit in fast memory then slow communication is eliminated between iterations of $j$.
The slow communication that remains is only that which is necessary to transmit pivot candidates between stages of refinement.

Unfortunately, writing and tuning CARRQR is nontrivial.
We implemented this algorithm and found that our version ran slightly slower than DGEQP3 on a shared memory parallel machine.
We believe this was principally due to inefficient parallelization of the final stages of refinement.
Our implementation assigned each column subset to a different processor which then worked independently to produce candidates.
This approach was attractive because it did not require communication between processors to complete each subset.
However, despite communication efficiency, this technique can only engage as many processors as there are column subsets.
During the final stages of refinement most processors were left idle.
A second problem with our approach occurred when the matrix was too tall.
In such cases it is not possible to select column subsets that are thin enough to fit in fast memory.
An efficient implementation would need alternative or additional workload-splitting tactics to use all processors at every stage of refinement.

As we discuss in the next section, the method we propose also gathers pivots into blocks which are then applied to the trailing matrix.
However, our method improves performance by reducing both communication and computational complexity needed to form a block of pivots.
This method is not incompatible with CARRQR.
It is possible that high-performance implementations on distributed memory machines will benefit from a hybrid selection scheme.

\section{Randomized sampling}

Randomized sampling is a computational tool that has recently gained traction in a number of applications in numerical linear algebra.
Random sampling reduces communication complexity via dimensional reduction while simultaneously maintaining high probability of safe error bounds on the approximations that follow.
This is the result of the well-known Johnson-Lindenstrauss Lemma~\cite{JL1984}.

Let $a_j$ represent the $j$\textsuperscript{th} column of $A$ for $j=1, 2, \ldots, n$.
Using a randomized $\ell \times m$ compression matrix $\Omega$ with unit-variance Gaussian Independent Identically Distributed (GIID) elements, we can construct sample columns $b_j = \Omega a_j$ which have 2-norm expectation values and variance
\[
    \mathbb{E}\left(\norm{b_j}_2^2\right) = \ell \norm{a_j}_2^2
    \quad\mbox{ and }\quad
    \mathbb{V}\left(\norm{b_j}_2^2\right) = 2 \ell \norm{a_j}_2^4
    \mbox{.}
\]
Furthermore, the probability of successfully detecting all column norms as well as all distances between columns within a relative error $\tau$ for $0 < \tau < \frac{1}{2}$ is bounded by
\[
    \texttt{Pr} \left(
        \left| \frac{\|b_j-b_i\|_2^2}{\ell \|a_j-a_i\|_2^2} - 1 \right|
        \leq \tau \right) \geq 1 - 2 e^{\frac{-\ell \tau^2}{4}(1 - \tau)}
\]
where $i=0,1, \ldots, n$, $a_0=0$, and $b_0=0$.

We use the sample matrix $B=\Omega A$ to select the column with largest approximate norm.
Subsequent columns are selected by continuing QRCP on the sample as seen in SSRQRCP, Algorithm~\ref{alg:ssrqrcp}, which we justify in Section~\ref{subsec:ssrqrcp}.
This reduces both communication and computation complexity associated with selecting a block of $b$ pivots by a factor of $\ell / m$.
More significantly, if the sample matrix $B$ fits in fast memory then slow communication between consecutive pivot decisions is eliminated within each block iteration.

Once $b$ pivots have been selected from the sample matrix $B$ the corresponding columns of $A$ are permuted and processed all at once as is done in DGEQRF.
We sacrifice knowing exact trailing norms when pivot decisions are made and settle for approximations that have extremely high probability of detecting magnitudes needed to reveal rank or construct reliable low-rank approximations.
Like DGEQRF, the remaining performance bottleneck is due to matrix multiplication needed to perform block reflections.
As such, Randomized QRCP (RQRCP) algorithms satisfy the performance standard of level-3 BLAS kernels.
Algorithm~\ref{alg:ssrqrcp} was first implemented in a term project in a course on random sampling at UC Berkeley in April 2014~\cite{Duersch2014}.
It is the simplest example of this approach.
P.G. Martinsson independently developed sample-based column pivoting in parallel work~\cite{Martinsson2015}.

SSRQRCP is acceptable for very-low-rank approximations in which the required sample size is small enough to maintain communication efficiency.
For larger approximations we will resort to a more comprehensive algorithm including a sample update formulation that subsumes this version.
However the single-sample algorithm illuminates the performance advantage gained from this approach, so we examine it first.

\begin{algorithm}
\caption{SSRQRCP - Single-Sample Randomized QRCP.}
\label{alg:ssrqrcp}
\begin{algorithmic}[1]
\Require
\Statex $A$ is $m \times n$.
\Statex $k$ the desired approximation rank.
$k \ll \min(m,n)$.
\Ensure
\Statex $Q$ is $m \times m$ orthogonal matrix in the form of $k$ reflectors.
\Statex $R$ is $k \times n$ truncated upper trapezoidal matrix.
\Statex $P$ is $n \times n$ permutation matrix such that $AP \approx Q(:,1:k)R$.
\MFunction {$[Q,R,P]$}{ssrqrcp}{$A, k$}
\State Set sample rank $l=k+p$ needed for acceptable sample error.

\State Generate random $l \times m$ GIID compression matrix $\Omega$.
\State Form the sample $B=\Omega A$.
\State \textbf{Get $k$ column pivots from sample, $[Q_b,R_b,P]=\texttt{qrcp}(B)$.}
\State \textbf{Apply permutation $A^{(1)}=A^{(0)}P$.}
\State Construct $k$ reflectors from new leading columns, $[Q,R_{11}]=\texttt{qr}(A^{(1)}(\texttt{:,1:k}))$.
\State Finish $k$ rows of $R$ in remaining columns, $R_{12}=Q(\texttt{:,1:k})^T A^{(1)}(\texttt{:,k+1:n})$.
\EndMFunction
\end{algorithmic}
\end{algorithm}

Although a GIID compression matrix $\Omega$ is more computationally expensive than some alternatives, it only contributes a small fraction to the total time.
For example, randomization comprises 7\% of processing time on a 12000-by-12000 matrix with $k=32$ and $\ell=40$ on a 24-core test machine.
Furthermore, this fractional contribution becomes even smaller in algorithms that follow.
Therefore we believe optimization of $\Omega$ is premature at this stage.
It is also a robust choice because it is both dense and invariant in distribution under independent orthogonal transformations.
These characteristics support our analysis.

\subsection{Sample norm distribution updates during SSRQRCP}
\label{subsec:ssrqrcp}

The 2-norms-squared of the columns of the sample matrix $B = \Omega A$ are proportional to the 2-norms-squared of the corresponding columns of $A$ with a constant of proportionality following the Chi-squared distribution.
Let $a$ be any particular column of $A$.
We can represent $a$ as a unit-vector $q_a$ using $a = q_a \norm{a}_2$.
Furthermore, let $Q_{\bot a}$ be an orthogonal complement of $q_a$ so that $Q_a = [q_a \enspace Q_{\bot a}]$ is an orthogonal matrix.
The $\ell \times m$ GIID compression matrix $\Omega$ can be represented in this basis as
\[
    \Omega = \left[ \begin{array}{cc}
		w_1 & W_2 \\
		\end{array} \right]
		\left[ \begin{array}{c}
			q_a^T \\
			Q_{\bot a}^T \\
		\end{array} \right]
	\quad\mbox{which gives a sample column}\quad
	b = \Omega a = w_1 \norm{a}_2
	\mbox{.}
\]
GIID matrices are invariant in distribution under independent orthogonal transformations.
If $Q_a$ is independent of $\Omega$ then we may regard both the $\ell$-element column $w_1$ and the $\ell \times (m-1)$ matrix $W_2$ as GIID matrices.
The 2-norm-squared of $w_1$ must therefore follow the Chi-squared distribution with $\ell$ degrees of freedom.
That is,
\[
	\|b\|_2^2 = x \|a\|_2^2
	\quad\mbox{where}\quad
	\rho_\ell(x) = \frac{ \left( \frac{x}{2} \right)^{\frac{\ell}{2}} e^{-\frac{x}{2}} }{x \Gamma(\frac{\ell}{2})}
	\quad\mbox{giving}\quad
	\mathbb{E}(x) = \ell
	\quad\mbox{and}\quad
	\mathbb{V}(x) = 2 \ell.
\]

If QRCP is performed on the sample matrix $B$, then at iteration $j$ we can examine $B$ as a partial factorization.
Let $P_j$ be the aggregate permutation so far.
Represent the accumulated orthogonal transformations applied to $B$ as $Q_{b,j}$ with corresponding intermediate triangular factor $S$ as shown below.
We could also construct a partial factorization of $A$ using the same pivots that were applied to $B$.
The corresponding factors of $A$ are $Q_j$ and $R$.
\[
    B P_j = Q_{b,j} \begin{bmatrix}
        S_{11}^{(j)} & S_{12}^{(j)} \\
        0            & S_{22}^{(j)} \\
    \end{bmatrix}
    \quad\mbox{ and }\quad
    A P_j = Q_j \begin{bmatrix}
        R_{11}^{(j)} & R_{12}^{(j)} \\
        0            & R_{22}^{(j)} \\
    \end{bmatrix}
    \mbox{.}
\]
Both $S_{11}^{(j)}$ and $R_{11}^{(j)}$ are upper triangular.
$\Omega$ can then be expressed as elements $W$ in the bases given by $Q_{b,j}$ and $Q_j$:
\[
    \Omega = Q_{b,j} \begin{bmatrix}
    	W_{11}^{(j)} & W_{12}^{(j)} \\
    	W_{21}^{(j)} & W_{22}^{(j)} \\
    \end{bmatrix} Q_j^T
    \mbox{.}
\]
Noting that $B P_j = \Omega A P_j$, we have
\begin{equation}
    \begin{bmatrix}
        S_{11}^{(j)} & S_{12}^{(j)} \\
        0            & S_{22}^{(j)} \\
    \end{bmatrix} =
    \begin{bmatrix}
      W_{11}^{(j)} R_{11}^{(j)} & W_{11}^{(j)} R_{12}^{(j)}+W_{12}^{(j)} R_{22}^{(j)} \\
      W_{21}^{(j)} R_{11}^{(j)} & W_{21}^{(j)} R_{12}^{(j)}+W_{22}^{(j)} R_{22}^{(j)} \\
    \end{bmatrix}\mbox{.}
    \label{eq:sample}
\end{equation}

If $S_{11}^{(j)}$ is non-singular then both $W_{11}^{(j)}$ and $R_{11}^{(j)}$ are also non-singular.
It follows that $W_{11}^{(j)} = S_{11}^{(j)} R_{11}^{(j)-1}$ is upper triangular and $W_{21}^{(j)}=0$.
In other words, we have implicitly formed a QR factorization of $\Omega Q_j$ using the same orthogonal matrix $Q_{b,j}$ that corresponds to $BP_j$.
Finally, the trailing matrix in the sample simplifies to $S_{22}^{(j)}=W_{22}^{(j)}R_{22}^{(j)}$, which is a sample of the trailing matrix $R_{22}^{(j)}$ using the compression matrix $W_{22}^{(j)}$.
We note, however, that $W_{22}^{(j)}$ is not GIID which we discuss in detail in Section~\ref{subsec:normtrunc}. 
This formulation justifies using column norms of $S_{22}^{(j)}$ to approximate column norms of $R_{22}^{(j)}$ when we select the $(j+1)$\textsuperscript{st} pivot.
A full block of pivots can be selected without interleaving any references to $A$ or $R$ memory.

\subsection{Sample bias}
The bias of an estimator is the difference between the expected value of the estimator and the true value of the quantity being estimated.
In this case, sample column norms are used to estimate true column norms.
Unfortunately, using the sample to make pivot decisions, which is indeed the entire purpose of random sampling, results in two forms of bias.
Since these biases could potentially interfere with the efficacy of the algorithms we propose, we attempt to clarify and quantify their effects.

The first form of bias is \textit{post hoc selection}.
This can be understood with a simple example.
Suppose we flip a coin a few times.
If the coin is fair, we expect half of the flips to be heads.
Now suppose this experiment is repeated several times and we select the trial that produced heads most frequently.
We now expect more than half of the flips to be heads in the selected trial even though the same coin was used.
The selected trial becomes biased because we have removed the possibility of observing a lower frequency than the alternatives.
As a consequence, the expectation value increases.

Similarly, we use the sample matrix to estimate true norms in order to permute the largest column to the front.
The selected sample column will be more likely to exhibit an unusually large norm due to the fact that it was specifically selected as the maximum.
This bias is present even if a new compression matrix is used for every column selection.
This bias is most pronounced when many columns are nearly tied for having the largest norm.
In such situations, sample noise out-weighs the true distinctions between columns.

\begin{figure}[!ht]
	\centering
	\includegraphics[width=0.80\textwidth]{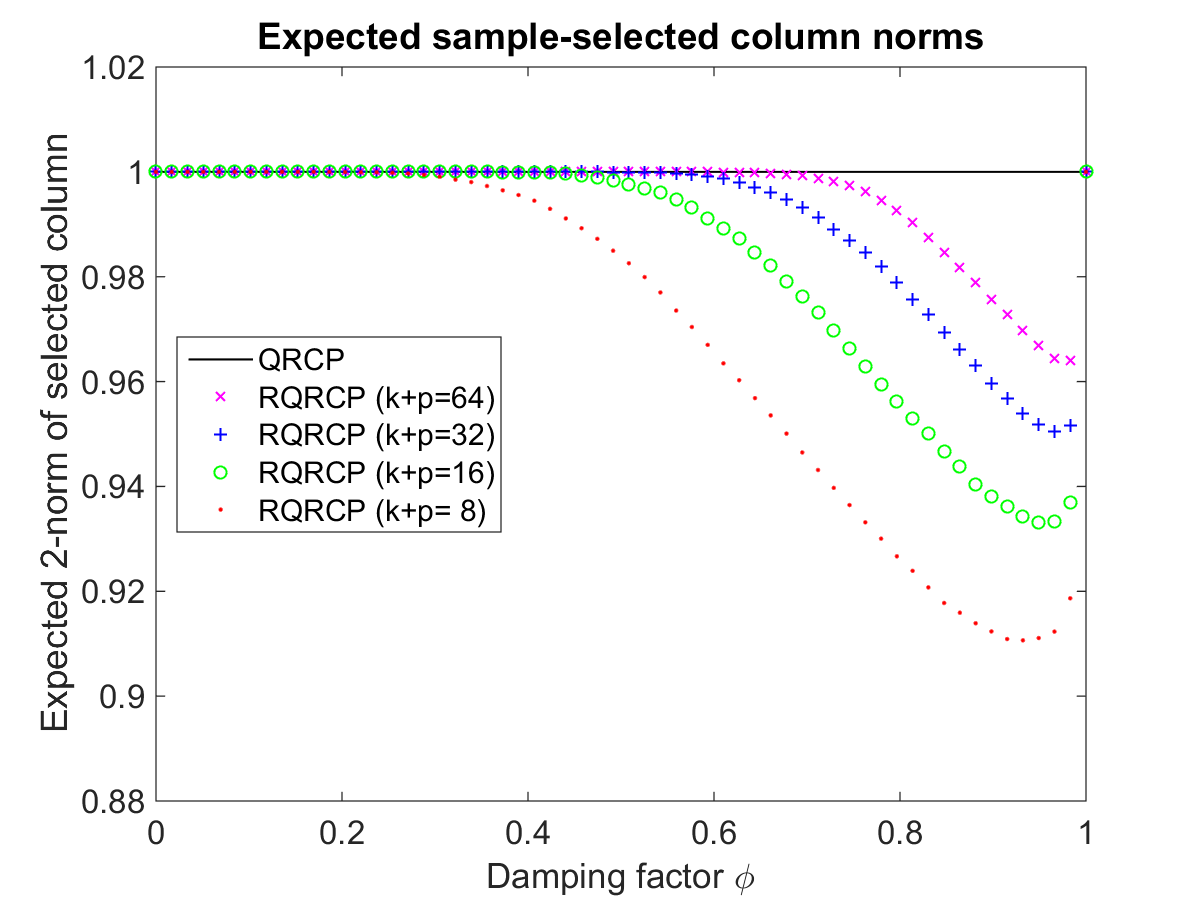}
	\caption{Each value of $\phi$ corresponds to a matrix in which column $j$ has $\norm{a_j}_2 = \phi^{j-1}$.
The damping factor $\phi$ determines how quickly the column norms descend.
We construct corresponding sample columns $b_j = \Omega a_j$ where $\Omega$ has rank $k+p$ and take $j_{\texttt{max}}$ for which $\norm{b_{j_{\texttt{max}}}}_2$ is maximal.
The expectation value $\mathbb{E}(\norm{a_{j_{\texttt{max}}}}_2)$ is plotted.
Greater expectation values indicate better selection performance.
In comparison, $QRCP$ always produces the maximum column.
This shows that when norms descend quickly (shown by smaller values of $\phi$) sample-selection produces nearly optimal norms.
Unsurprisingly, increasing the sample rank improves selections.
Also note that even the worst case expectation value we observe is over 90\% of the optimum.}
	\label{fig:selectExpect}
\end{figure}

The numerical experiment shown in Figure~\ref{fig:selectExpect} computes the expected 2-norm of a column selected by this method.
The matrix $A$ is constructed to have orthogonal columns with scaled 2-norms according to a decaying exponential so that column $j$ has 2-norm $\norm{a_j}_2 = \phi^{j-1}$ for $\phi \in [0, 1]$.
 When the damping factor $\phi$ is close to $1$ we see that sub-optimal choices are more likely, however they do little damage.
When several columns are nearly tied for first place, any of them would serve well as the next pivot.
We are more concerned with suppressing damaging selections.
For rank-revealing applications or low-rank approximations, it is much more important to avoid picking a column that is an order of magnitude smaller than alternatives.
This is the scenario in which the sample performs well.

\subsection{Norm distribution truncation}
\label{subsec:normtrunc}

The second form of bias arises after partial factorization.
If the sample had not been used to make pivot decisions then $W_{22}^{(j)}$ would be equivalent in distribution to a new GIID matrix.
That is because $Q_j$ would be formed from columns of $A$, which would be independent of $\Omega$.
Similarly, $Q_{b,j}$ would be formed to triangularize the leading columns of $\Omega Q_j$, which would be independent of subsequent columns.
As such, the trailing column norms in $B$ would remain unbiased predictors for the trailing column norms in $A$.
However, that is not the case; prior column selections depended on the sample outcome.
As a consequence, the remaining lower-right partition of the sample factorization is no longer equivalent in distribution to a GIID compression.
We now examine the new estimator probability distribution in detail.

Suppose we have a partial factorization of the sample matrix after $j$ rows and columns are complete:
\[
	Q_{b,j}^{T} B P_j =
		\begin{blockarray}{[cccccc]}
			s_{1,1} & \cdots & s_{1,j} & s_{1,j+1}     & \cdots & s_{1,n} \\
			0       & \ddots & \vdots  &               & \vdots &         \\
			0       &        & s_{j,j} & s_{j,j+1}     & \cdots & s_{j,n} \\
			0       & \cdots & 0       & \hat{b}_{j+1} & \cdots & \hat{b}_n \\
		\end{blockarray}
		\mbox{.}
\]
Let $i = 1, \ldots, j$ be any completed row index above.
Likewise, remaining columns are indexed by $j' = j+1, \ldots, n$.
At a previous iteration $i$, QRCP had pivoted the largest trailing norm to the leading edge which then became the new diagonal element.
That means
\[
	s_{i,i}^2 \ge \sum_{k=i}^j s_{k,j'}^2 + \norm{\hat{b}_{j'}}_2^2
	\quad\mbox{ for all }\quad
	i \le j
	\quad\mbox{ and }\quad
	j' > j.
\]
In order for the estimator $\norm{\hat{b}_{j'}}_2^2$ to be consistent with the $i$\textsuperscript{th} pivot decision, it must be bounded from above.
As before, we write the sample trailing norm as a factor of the true trailing norm $\norm{\hat{b}_{j'}}_2^2 = x_{j'} \norm{\hat{a}_{j'}}_2^2$ which gives
\[
	x_{j'} \le \tau_{i,j'}
	\quad\mbox{ where }\quad
	\tau_{i,j'} = \frac{1}{\norm{\hat{a}_{j'}}_2^2} \left( s_{i,i}^2 - \sum_{k=i}^j s_{k,j'}^2 \right)
	\mbox{.}
\]
For each remaining column $j'>j$, the controlling upper bound is obtained by minimizing over all previous rows $i \le j$:
\[
     x_{j'} \le \tau_{j'}
     \quad\mbox{ where }\quad
     \tau_{j'} = \min_{i=1, \ldots, j} \tau_{i,j'}
     \mbox{.}
\]

\begin{theorem}
Given a partial factorization as above in which $\tau_{j'}$ is known for any remaining column $j'>j$, the trailing 2-norm-squared of the sample corresponds to a truncated Chi-squared distribution with $\ell - j$ degrees of freedom.
The probability density function is
\[
	\rho_{\ell-j,\tau_{j'}}(x_{j'}) =
	\begin{cases}
		\frac{ \left( \frac{x_{j'}}{2} \right)^{\frac{\ell-j}{2}} e^{-\frac{x_{j'}}{2}} }{x_{j'} \gamma(\frac{\ell-j}{2}, \frac{\tau_{j'}}{2})} & x_{j'} \leq \tau_{j'} \\
		0 & x_{j'} > \tau_{j'} \\
	\end{cases}.
\]
The normalization factor $\gamma()$ in the formulas above is the lower-incomplete gamma function $\gamma(a,z) = \int_0^z t^{a-1} e^{-t} dt$.
\end{theorem}

\begin{proof}
The conditional probability density function for $x_{j'}$ given a particular permutation $P^*$ satisfies Bayes' theorem in the form
\[
	\rho(x_{j'} | P^*) \texttt{Pr}( P^*) =
	\texttt{Pr} ( P^* | x_{j'} ) \rho(x_{j'})
\]
where we have a discrete probability function in the variable $P^*$ and continuous distributions in the variable $x_{j'}$.
If the controlling upper bound above is not satisfied then QRCP would not have produced the permutation $P^*$.
It follows that $\texttt{Pr} ( P^* | x_{j'} )=0$ for $x_{j'}>\tau_{j'}$.
Conversely, as long as $x_{j'}<\tau_{j'}$ the particular value of $x_j$ does not affect any branch decisions in QRCP thus far.
Within this region $\texttt{Pr} ( P^* | x_{j'} )$ is a constant independent of $x_{j'}$.
It follows
\[
	\rho(x_{j'} | P^*) \propto
	\begin{cases}
	  \rho(x_{j'}) & x_{j'} < \tau_{j'} \\
		        0 & x_{j'} > \tau_{j'} \\
	\end{cases}
	\mbox{.}	
\]
Normalizing to unit cumulative probability gives the stated result \quad\end{proof}

The truncated Chi-squared distribution above has the expectation value
\[
	\mathbb{E}(x_{j'}) = (\ell-j) \left(
		1 - \frac{e^{-\frac{\tau_{j'}}{2}}}{f(\frac{\ell-j}{2},\frac{\tau_{j'}}{2})}
	\right)
	\enskip\mbox{where}\enskip
	f(s,z) = \frac{s \gamma(s, z)}{z^s}
	\enskip\mbox{and}\enskip
	\lim_{z \to 0} f(s,z) = 1.
\]
The expectation value is monotonically increasing from zero at $\tau_{j'}=0$ and approaches that of the complete Chi-squared distribution with $\ell - j$ degrees of freedom as $\tau_{j'}$ grows large.

This only affects pivot selection if the leading candidate is biased down more than alternatives.
Let us consider a scenario in which two columns $j_1$ and $j_2$ have equivalent trailing norms $\norm{\hat{a}_{j_1}}_2 = \norm{\hat{a}_{j_2}}_2 = \alpha$.
Then the expectation values of the trailing sample column norms-squared are
\[
	\mathbb{E}(\norm{\hat{b}_{j_1}}_2^2) = \alpha^2 \mathbb{E}(x_{j_1})
	\quad\mbox{and}\quad
	\mathbb{E}(\norm{\hat{b}_{j_2}}_2^2) = \alpha^2 \mathbb{E}(x_{j_2}).
\]
This is informative because both cases grow with the same monotonically increasing function of $\tau_{j_1}$ and $\tau_{j_2}$ respectively.
For simplicity, let us assume that the same row index $i$ gives the controlling upper bounds forming $\tau_{j_1}$ and $\tau_{j_2}$.
If $\sum_{k=i}^j s_{k,{j_1}}^2 > \sum_{k=i}^j s_{k,{j_2}}^2$ then $\tau_{j_1} < \tau_{j_2}$ which means the sample is more likely to select column $j_2$.
In other words, the sample is biased in favor of the column that had smaller components in the previously factored subspace.
In contrast, QRCP would not prefer one column over the other and it would simply select the column with the lower index.

It might be possible to construct an algorithm that accounts for decreased expectation values by either scaling or re-sampling when we are no longer confident that the sample adequately represents the true trailing matrix.
We also note that when all of the remaining trailing columns in the sample have relatively small norms in comparison to previously factored components---which occurs at gaps in the spectrum---then the cutoff thresholds are relatively large.
Thus the sample behaves more like an unbiased sample on remaining columns at spectral gaps.

\section{Sample updates}

The original sample matrix $B=\Omega A$ was constructed with rank $\ell = k + p$ where $k$ was the desired approximation rank and $p$ was additional padding required to ensure the sample error remained below an acceptable threshold.
Accordingly, we could only safely select a block of $b=k$ pivots from the sample.
As $k$ increases, however, it becomes inefficient to simply increase the sample rank $\ell$.
In the extreme case, a full decomposition would require a sample just as big as the original matrix which would eliminate the performance advantage of using the sample to select pivots.
If we require a decomposition with a larger rank than that which can be efficiently sampled and blocked, that is if the sample rank cannot exceed $\ell = b + p$, but we require $k>b$, then we need to update the sample matrix.
Doing so will allow the algorithm to continue selecting pivots after the first block is processed.
Martinsson's approach~\cite{Martinsson2015} applies a new compression matrix to the trailing columns to continue.
This version, Repeated-Sampling Randomized QRCP (RSRQRCP), is outlined in Algorithm~\ref{alg:rsrqrcp}.
We propose a sample update formulation that does not require multiplying the trailing matrix by a new compression matrix to continue.
As a result, the proposed update reduces level-3 BLAS communication in the overall factorization by at least one third.

The update formula we now derive is an extension of the implicit update mechanism described in the prior section.
Note that both Algorithm~\ref{alg:rqrcp}, RQRCP, and Algorithm~\ref{alg:trqrcp}, TRQRCP, that follow will proceed in blocks of pivots.
Bracket superscripts denote results of a computation that occurred on the indicated block-iteration.
At entry to the first block-iteration the sample is represented $B^{[0]}=\Omega^{[0]} A^{[0]}$ where $A^{[0]}$ represents the original matrix.
Continuing from Equation~\ref{eq:sample} at the end of block-iteration $J$ the sample array is left in the transformed state:
\begin{equation}
	\begin{bmatrix}
		S_{11}^{[J]} & S_{12}^{[J]} \\
    0            & S_{22}^{[J]} \\
  \end{bmatrix}
  =
  \begin{bmatrix}
  	W_{11}^{[J]} & W_{12}^{[J]} \\
    0            & W_{22}^{[J]} \\
  \end{bmatrix}
  \begin{bmatrix}
  	R_{11}^{[J]} & R_{12}^{[J]} \\
    0            & A^{[J]} \\
  \end{bmatrix}.
  \label{eq:sampleR}
\end{equation}

$S_{11}^{[J]}$ is the leading upper triangle from the partial factorization of the sample.
$S_{22}^{[J]}$ gives the trailing columns of the sample.
Likewise $R_{11}^{[J]}$ and $A^{[J]}$ respectively give the leading upper triangle and trailing columns that would be obtained by factorizing the original matrix with the same pivots.
By absorbing the transformations $Q_b^{[J]T}$ and $Q^{[J]}$ into $\Omega^{[J]}$ we obtained an effective compression matrix $W_{22}^{[J]}$ which had already been implicitly applied to the trailing columns: $S_{22}^{[J]} = W_{22}^{[J]} A^{[J]}$.
The difficulty is $S_{22}^{[J]}$ only has rank $p$.
In order to construct a rank $\ell = b + p$ sample of the trailing matrix $A^{[J]}$ we need to include $W_{12}^{[J]}$ in the updated compression matrix:
\[
    \Omega^{[J]} =
    \begin{bmatrix}
      W_{12}^{[J]} \\
      W_{22}^{[J]} \\
    \end{bmatrix}
    \quad
    \mbox{giving}
    \quad
    B^{[J]} = \Omega^{[J]} A^{[J]} =
    \begin{bmatrix}
    	S_{12}^{[J]} - W_{11}^{[J]} R_{12}^{[J]} \\
      S_{22}^{[J]} \\
    \end{bmatrix}.
\]
In other words, the new compression matrix $\Omega^{[J]}$ is simply $Q_b^{[J]T}\Omega^{[J-1]}Q^{[J]}$ with the leading $b$ columns removed.
This new compression matrix does not need to be explicitely formed or applied to the trailing columns $A^{[J]}$.
Instead we form the result implicitly by removing $W_{11}^{[J]} R_{12}^{[J]}$ from $S_{12}^{[J]}$.
Both $R_{11}^{[J]}$ and $R_{12}^{[J]}$ will be computed in blocked matrix multiply operations using the previous $b$ pivots of $A$.
Since $W_{11}^{[J]}$ can then be recovered from $S_{11}^{[J]}$ we can avoid any direct computations on $\Omega$.
We only need to update the first $b$ rows of $B$ which gives us the sample update formula
\begin{equation}
    \left[ \begin{array}{c}
        B_1^{[J]} \\
        B_2^{[J]} \\
    \end{array} \right]
    = \left[ \begin{array}{c}
        S_{12}^{[J]} - S_{11}^{[J]} R_{11}^{[J]-1} R_{12}^{[J]} \\
        S_{22}^{[J]} \\
      \end{array} \right].
      \label{eq:update1}
\end{equation}

Algorithm~\ref{alg:rqrcp}, RQRCP, describes how the sample update formula is used to produces a full factorization.
Of course this introduces a more complicated form of bias in the updated sample.
Again, if the sample had not been used to select pivots then the updated compression matrix would remain GIID for the same reasons as before.
Analysis would require us to exclude sample outcomes that would be inconsistent with pivot decisions from all previous block iterations.
We do not attempt such analysis at this time, but experimental results in Section~\ref{subsec:decompquality} show that RQRCP performs as well as RSRQRCP on the matrices we tested.

\section{Randomized QRCP}

Full randomized QR with column pivoting (RQRCP) can be structured as a modification to blocked level-3 BLAS QR.
The algorithm must simply interleave processing blocks of reflectors with permutations obtained from the sample matrix.
This is described in detail in Algorithm~\ref{alg:rqrcp}, RQRCP.
For comparison, we outline a version that does not employ sample updates in Algorithm~\ref{alg:rsrqrcp}, Repeated-Sampling RQRCP (RSRQRCP).

\begin{algorithm}[ht]
	\caption{Randomized QR with Column Pivoting, RQRCP}
	\label{alg:rqrcp}
	\begin{algorithmic}[1]
		\Require
			\Statex $A$ is $m \times n$.
		    \Statex $k$ is the desired factorization rank.
$k \le \min{(m,n)}$.
		\Ensure
		    \Statex $Q$ is $m \times m$ orthogonal matrix in the form of $k$ reflectors.
		    \Statex $R$ is $k \times n$ upper trapezoidal (or triangular) matrix.
		    \Statex $P$ is $n \times n$ permutation matrix such that $AP \approx Q(\texttt{:,1:k})R$.
		\MFunction {$[Q,R,P]$}{rqrcp}{$A,k$}
		\State Set sample rank $\ell = b + p$ needed for acceptable sample error.
		\State Generate random $\ell \times m$ GIID matrix $\Omega^{[0]}$.
		\State Form the initial sample $B^{[0]} = \Omega^{[0]} A^{[0]}$.
		\MDo{J=1, 2, \ldots, $\frac{k}{b}$}
		    \State \textbf{Get $b$ column pivots from sample,} $[Q_b^{[J]},S^{[J]},P_b^{[J]}] = \texttt{qrcp}(B^{[J-1]},b)$.
		    \State \textbf{Permute $A^{[J-1]}$ and completed rows in $R$ with $P_b^{[J]}$.}
		    \State Construct $b$ reflectors, $[Q^{[J]},R_{11}^{[J]}] = \texttt{qr}(A^{[J-1]}P_b^{[J]}(\texttt{:,1:b}))$.
		    \State Finish $b$ rows, $R_{12}^{[J]}=Q^{[J]}(\texttt{:,1:b})^T A^{[J-1]}P_b^{[J]}(\texttt{:,b+1:end})$.
		    \State Update trailing matrix, $A^{[J]}=Q^{[J]}(\texttt{:,b+1:end})^T A^{[J-1]}P_b^{[J]}(\texttt{:,b+1:end})$.
		    \State \textbf{Update sample,} $B_{1}^{[J]}=S_{12}^{[J]}-S_{11}^{[J]}R_{11}^{[J]-1} R_{12}^{[J]}$ and $B_{2}^{[J]}=S_{22}^{[J]}$.
		\EndMDo
		\State $Q = Q^{[1]} Q^{[2]} \ldots Q^{[k/b]}$.
		\State $P = P_b^{[1]} P_b^{[2]} \ldots P_b^{[k/b]}$.
    	\EndMFunction
	\end{algorithmic}
\end{algorithm}

\begin{algorithm}[hb]
	\caption{Repeated-Sampling Randomized QRCP, RSRQRCP}
	\label{alg:rsrqrcp}
	\begin{algorithmic}[1]
		\Require
			\Statex $A$ is $m \times n$.
		    \Statex $k$ is the desired factorization rank.
$k \le \min{(m,n)}$.
		\Ensure
		    \Statex $Q$ is $m \times m$ orthogonal matrix in the form of $k$ reflectors.
		    \Statex $R$ is $k \times n$ upper trapezoidal (or triangular) matrix.
		    \Statex $P$ is $n \times n$ permutation matrix such that $AP \approx Q(\texttt{:,1:k})R$.
		\MFunction {$[Q,R,P]$}{rsrqrcp}{$A,k$}
		\State Set sample rank $\ell = b + p$ needed for acceptable sample error.
		\MDo{J=1, 2, \ldots, $\frac{k}{b}$}
		    \State \textbf{Generate $\ell \times (m-Jb)$ GIID matrix $\Omega^{[J-1]}$}
		    \State \textbf{Form sample $B^{[J-1]} = \Omega^{[J-1]} A^{[J-1]}$}
		    \State Get $b$ column pivots from sample, $[Q_b^{[J]},S^{[J]},P_b^{[J]}] = \texttt{qrcp}(B^{[J-1]},b)$.
		    \State Permute $A^{[J-1]}$ and completed rows in $R$ with $P_b^{[J]}$.
		    \State Construct $b$ reflectors, $[Q^{[J]},R_{11}^{[J]}] = \texttt{qr}(A^{[J-1]}P_b^{[J]}(\texttt{:,1:b}))$.
		    \State Finish $b$ rows, $R_{12}^{[J]}=Q^{[J]}(\texttt{:,1:b})^T A^{[J-1]}P_b^{[J]}(\texttt{:,b+1:end})$.
		    \State Update trailing matrix, $A^{[J]}=Q^{[J]}(\texttt{:,b+1:end})^T A^{[J-1]}P_b^{[J]}(\texttt{:,b+1:end})$.
		\EndMDo
		\State $Q = Q^{[1]} Q^{[2]} \ldots Q^{[k/b]}$.
		\State $P = P_b^{[1]} P_b^{[2]} \ldots P_b^{[k/b]}$.
    	\EndMFunction
	\end{algorithmic}
\end{algorithm}

When QRCP is applied to the sample matrix $B$, only a partial decomposition is necessary.
The second argument $b$ in the subroutine call $\texttt{qrcp}(B^{[J]},b)$ indicates that only $b$ column permutations are required.
Although the additional cost of processing all $\ell$ columns should be small, halting the computation early is a trivial modification.

After sample pivots have been applied to the array containing both $A$ and $R$, we perform QR factorization on the new leading $b$ columns of the trailing matrix.
Reflectors are then applied to the trailing matrix and we form the sample update $B^{[J]}$ to prepare for the next iteration.

\subsection{Truncated RQRCP avoiding trailing update}

The trailing matrix is usually not needed for low-rank approximations and the algorithm can be reformulated to run roughly twice as fast on large matrices, provided $k\ll\min(m,n)$.
This is accomplished by avoiding the trailing update which reduces large matrix multiplications by half.

The technique is analogous to the method Quintana-Ort\'{\i} \textit{et al.\@} used to halve level-2 BLAS operations in QRCP.
In their version of QRCP, all reflector inner products are computed, but rows and columns are only updated as needed.
In order to compute correct reflector inner products without having updated the trailing matrix, we need block reflector composition formulas.
\[
    (I - Y_1 T_1 Y_1^T)(I-Y_2 T_2 Y_2^T) = I - Y T Y^T
\]
where $Y$ and $T$ are partitioned
\[
    Y =
    \begin{bmatrix}
    	Y_1 & Y_2 \\
    \end{bmatrix}
    \quad\mbox{ and }\quad
    T =
    \begin{bmatrix}
    	T_1 & -T_1 Y_1^T Y_2 T_2 \\
    	0   & T_2 \\
    \end{bmatrix}
    \mbox{.}
\]
Corresponding reflector inner products $W^T = T^T Y^T A$ are
\[
    W^T =
    \begin{bmatrix}
        W_1^T \\
        W_2^T \\
    \end{bmatrix}
    \quad
    \mbox{with}
    \quad
    W_1^T = T_1^T Y_1^T A
    \quad
    \mbox{and}
    \quad
    W_2^T = T_2^T \left( Y_2^T A - (Y_2^T Y_1) W_1^T \right)
    \mbox{.}
\]
If these reflector inner products are stored, then we can construct any sub-matrix of the accumulated transformation $\hat{A}^{[J]} = A - Y^{[J]} W^{[J]T}$ as needed.
Columns that are selected by sample pivots are constructed just before becoming the next reflectors and corresponding rows of $R$ are constructed just before being used to update the sample.
The trailing-update-avoiding algorithm Truncated Randomized QR with Column Pivot (TRQRCP) is outlined in Algorithm~\ref{alg:trqrcp}.

\begin{algorithm}[htb]
	\caption{Truncated RQRCP without trailing update, TRQRCP}
	\label{alg:trqrcp}
	\begin{algorithmic}[1]
		\Require
			\Statex $A$ is $m \times n$.
			\Statex $k$ approximation rank.
$k\ll\min(m,n)$.
		\Ensure
		    \Statex $Q$ is $m \times m$ orthogonal matrix in the form of $k$ reflectors.
		    \Statex $R$ is $k \times n$ upper trapezoidal matrix.
		    \Statex $P$ is $n \times n$ permutation matrix such that $AP \approx Q(\texttt{:,1:k})R$.
		\MFunction {$[Q,R,P]$}{rqrcp}{$A,k$}
		\State Set sample rank $\ell = b + p$ needed for acceptable sample error.
		\State Generate random $\ell \times m$ GIID matrix $\Omega^{[0]}$.
		\State Form the initial sample $B^{[0]} = \Omega^{[0]} A{[0]}$.
		\MDo{J=1, 2, \ldots, $\frac{k}{b}$}
		    \State Obtain $b$ pivots from the sample, $[Q_b^{[J]},S^{[J]},P_b^{[J]}] = \texttt{qrcp}(B^{[J]},b)$.
		    \State Permute $A^{[J]} = A^{[J-1]}P_b^{[J]}$ as well as completed rows of $R$.
		    \State Permute prior inner products, $W_1^{[J]T}=W^{[J-1]T}P_b^{[J]}$.
		    \State \textbf{Construct selected columns, $\hat{A}_J$, from $A^{[J]} - Y^{[J-1]} W_1^{[J]T}$.}
		    \State Form reflectors $Y_2^{[J]}$ using $[Q^{[J]},R_{11}^{[J]}] = \texttt{qr}(\hat{A}_J)$.
		    \State \textbf{Form inner products, $W_2^{[J]} = T_2^{[J]T} ( Y_2^{[J]T} A^{[J]} - (Y_2^{[J]T} Y^{[J-1]}) W_1^{[J]T} )$.}
		    \State Augment $Y^{[J]} = [Y^{[J-1]} \enspace Y_2^{[J]}]$ and $W^{[J]} = [W_1^{[J]} \enspace W_2^{[J]}]$.
		    \State \textbf{Construct new rows of $R$ from $A^{[J]} - Y^{[J]} W^{[J]T}$.}
		    \State Update sample, $B_{1}^{[J]}=S_{12}^{[J]}-S_{11}^{[J]}R_{11}^{[J]-1} R_{12}^{[J]}$ and $B_{2}^{[J]}=S_{22}^{[J]}$.
		\EndMDo
		\State $Q = Q^{[1]} Q^{[2]} \ldots Q^{[k/b]}$.
		\State $P = P_b^{[1]} P_b^{[2]} \ldots P_b^{[k/b]}$.
    	\EndMFunction
	\end{algorithmic}
\end{algorithm}

\section{Approximation of truncated SVD}

TRQRCP naturally extends to an approximation of the truncated Singular Value Decomposition (SVD).
This follows the QLP method proposed by Stewart~\cite{journals/siamsc/Stewart99}.
The QLP decomposition proceeds by first applying QRCP to obtain $A P_0 = Q_0 R$.
Then the right triangular matrix $R$ is factored again using an LQ factorization $P_1 R = L Q_1$ where row-pivoting is an optional safeguard (otherwise $P_1=I$).
This gives the factored form $A=(Q_0 P_1^T) L (Q_1 P_0^T)$.
The diagonal elements of $L$ give a very good approximation of the singular values of $A$.
Analysis is done by Huckaby and Chan~\cite{journals/na/HuckabyC03}.

The approximate truncated SVD proposed here simply applies low-rank versions of the steps in QLP.
The rank-$k$ approximation that results is exactly the same as the truncated approximation that would be obtained if QLP had been processed to completion using RQRCP---without secondary row-pivoting---and then truncated to a rank-$k$ approximation.

We begin by using TRQRCP to produce $k$ left reflectors to obtain the initial left orthogonal matrix $U^{(0)}$.
Results are simultaneously compared to what would have been obtained by full RQRCP-based QLP:
\[
	A P^{(0)} \approx
  	\begin{bmatrix} U^{(0)}_1 & U^{(0)}_2 \end{bmatrix}
  	\begin{bmatrix} R^{(0)}_{11} & R^{(0)}_{12} \\ 0 & 0 \\ \end{bmatrix}
\]
\[
    \wide{versus}\quad
    A P^{(0*)} =
  	\begin{bmatrix} U^{(0)}_1 & U^{(0*)}_2 \end{bmatrix}
  	\begin{bmatrix} R^{(0)}_{11} & R^{(0*)}_{12} \\ 0 & R^{(0*)}_{22} \\ \end{bmatrix}.
\]
The additional pivoting produced from the full factorization is denoted by the asterisk.
Clearly the first $k$ pivots in $P^{(0)}$ and corresponding reflectors in $U^{(0)}$ are the same.
The corresponding rows in $R^{(0)}$ are also the same modulo additional column permutations.
We can reverse these permutations to construct the $k \times n$ matrix $Z^{(0)}=R^{(0)} P^{(0)T}$:
\[
	A \approx
  	\begin{bmatrix} U^{(0)}_1 & U^{(0)}_2 \end{bmatrix}
    \begin{bmatrix} Z^{(0)}_{11} & Z^{(0)}_{12} \\ 0 & 0 \\ \end{bmatrix}
  \quad\mbox{versus}\quad
  A =
  	\begin{bmatrix} U^{(0)}_1 & U^{(0*)}_2 \end{bmatrix}
  	\begin{bmatrix} Z^{(0)}_{11} & Z^{(0)}_{12} \\ Z^{(0*)}_{21} & Z^{(0*)}_{22} \\ \end{bmatrix}.
\]

Taking the LQ factorization from $Z^{(0)}$, $L^{(1)} V^{(1)T} = Z^{(0)}$, instead of from $R^{(0)}$ simply absorbs the permutation $P^{(0)T}$ into the definition of $V^{(1)T}$.
That gives
\[
	A \approx
  	\begin{bmatrix} U^{(0)}_1 & U^{(0)}_2 \end{bmatrix}
    \begin{bmatrix}
    	L^{(1)}_{11} & 0 \\
      0            & 0 \\
    \end{bmatrix}
    \begin{bmatrix}
    	V^{(1)T}_1 \\
      V^{(1)T}_2 \\
    \end{bmatrix}
\]
\[
  \mbox{versus}\quad
  A =
  	\begin{bmatrix} U^{(0)}_1 & U^{(0*)}_2 \end{bmatrix}
  	\begin{bmatrix}
    	L^{(1)}_{11}  & 0             \\
    	L^{(1*)}_{21} & L^{(1*)}_{22} \\
    \end{bmatrix}
    \begin{bmatrix}
    	V^{(1)T}_1 \\
      V^{(1*)T}_2 \\
    \end{bmatrix}.
\]
For consistency with the form that follows, we could label the $k \times k$ connecting matrix $X^{(0)}=L^{(1)}_{11}$.
We return to the connecting matrix after explaining the rest of the algorithm.

The leading $k$ reflectors in $V^{(1)}$ and $V^{(1*)}$ are identical because they are only based on the leading $k$ rows of $Z^{(0)}$ provided no secondary row-pivoting is considered.
At this point, the rank-$k$ approximation of RQRCP-based QLP would require $L^{(1*)}_{21}$ which is unknown.
Fortunately, the leading $k$ columns of $U^{(0*)} L^{(1*)}$ can be reconstructed with one matrix multiply.
We label this $m \times k$ matrix $Z^{(1)}$.
\[
	Z^{(1)} = A V^{(1)}_1 =
  	\begin{bmatrix} U^{(0)}_1 & U^{(0*)}_2 \end{bmatrix}
    \begin{bmatrix}
      L^{(1)}_{11}  \\
      L^{(1*)}_{21} \\
    \end{bmatrix}
    \quad\mbox{in both cases.}
\]
This is QR-factorized $U^{(1)} X^{(1)} = Z^{(1)}$ to produce the approximation:
\begin{equation}
	A \approx
		U^{(1)}
  	\begin{bmatrix}
    	X^{(1)} & 0 \\
      0       & 0 \\
    \end{bmatrix}
    V^{(1)^T}.
    \label{eq:tuxv}
\end{equation}

Further iterations could be computed to produce subsequent $k \times k$ connection matrices $X^{(2)}$, $X^{(3)}$, etc.
which would flip between upper triangular and lower triangular forms.
To do this, one would simply multiply the leading rows of $U^T$ or columns of $V$ on the left and right of $A$ respectively.
This is outlined in Algorithm~\ref{alg:tuxv}, TUXV.
The leading singular values of $A$ are approximated on the diagonals of $X^{(j)}$, however since the connection matrix is small it would be feasible to obtain slightly better approximations by taking the SVD of $X^{(j)}$.
One could also insert mechanisms to iterate until a desired level of convergence is obtained, however as Stewart observed only one QRCP-LQ iteration is needed to produce a reasonable approximation of the SVD.
Note a subtle point of possible confusion is that by setting $j_{\mbox{\scriptsize{max}}}=1$ our algorithm might appear to produce a truncated approximation from the sequence RQRCP-LQ-QR.
It is true that the diagonal elements in $X$ correspond to that sequence, however the resulting factorization is equivalent to what would be obtained by keeping only the leading columns of L after RQRCP-LQ.
The final QR factorization simply extracts an orthogonal basis $U$.
In the next section, we test the performance of TUXV with $j_{\mbox{\scriptsize{max}}}=1$ for both timing and quality experiments.

\begin{algorithm}[htb]
	\caption{Approximation of truncated SVD, TUXV}
	\label{alg:tuxv}
	\begin{algorithmic}[1]
		\Require
			\Statex $A$ is $m \times n$ matrix to approximate.
			\Statex $k$ is approximation rank.
$k\ll\min(m,n)$.
			\Statex $j_{\mbox{\scriptsize{max}}}$ is number of LQ-QR iterations.
We set $j_{\mbox{max}}=1$.
		\Ensure
		    \Statex $U$ is orthogonal $m \times m$ matrix.
		    \Statex $V$ is orthogonal $n \times n$ matrix.
		    \Statex $X$ is $k \times k$ upper or lower triangular matrix.
		    \Statex $A \approx U(:,1:k) X V(:,1:k)^T$.

		\MFunction {$[U,X,V]$}{tuxv}{$A,k,\tau,j_{\mbox{\scriptsize{max}}}$}
		    \State \textbf{TRQRCP-Factorize,} $[U^{(0)},R^{(0)},P^{(0)}]=\texttt{trqrcp}(A,k)$.
		    \State Restore original column order, $Z^{(0)}=R^{(0)}P^{(0)T}$.
			\State \textbf{LQ-Factorize,} $[V^{(1)},X^{(0)T}] = \texttt{qr}(Z^{(0)T})$.
		    \MDo{$j=1,3,5,\ldots$}
		        \State \textbf{Multiply,} $Z^{(j)}=A V^{(j)}(\texttt{:,1:k})$.
		        \State \textbf{QR-Factorize,} $[U^{(j+1)},X^{(j)}] = \texttt{qr}(Z^{(j)})$.
		        \State If $j=j_{\mbox{\scriptsize{max}}}$ then break.
		        \State \textbf{Multiply,} $Z^{(j+1)}=U^{(j+1)}(\texttt{:,1:k})^T A$.
    			\State \textbf{LQ-Factorize,} $[V^{(j+2)},X^{(j+1)T}] = \texttt{qr}(Z^{(j+1)T})$.
		        \State If $j+1=j_{\mbox{\scriptsize{max}}}$ then break.
			\EndMDo
		\EndMFunction
	\end{algorithmic}
\end{algorithm}
\clearpage

\section{Experiments}

Our first Fortran version of RQRCP used simple calls to BLAS and LAPACK subroutines without directly managing workloads among available cores.
Library implementations of BLAS and LAPACK subroutines automatically distribute the computation to available cores using OpenMP.
Although we knew RQRCP should have nearly the same complexity as blocked QR, that version did not compete well with library calls to the LAPACK subroutine \textbf{dgeqrf}, the level-3 BLAS QR factorization.
In order to provide a convincing demonstration of the efficiency of RQRCP, it was necessary to carefully manage workloads using OpenMP within each phase of the main algorithm.
The following experiments show that the subroutines we have proposed, RQRCP and TRQRCP, can be written to require substantially less computation time than the optimized QRCP implementation \texttt{dgeqp3} available through Intel's Math Kernel Library.
The pivots that result from randomized sampling are not the same as those obtained from QRCP, however we claim that they are of similar quality.
In order to compare factorization quality, we construct sequences of partial factorizations and compute the corresponding truncated approximation error in the Frobenius norm.
These simple experiments show that the pivots obtained from RQRCP yield partial factorizations that are nearly indistinguishable to those obtained from QRCP.
We emphasize that RQRCP is not intended to overcome the well-known pitfalls of QRCP in computing the precise rank of difficult test matrices.
It simply produces comparable results at a much lower cost.

Our low-rank approximation experiments include TUXV with $j_{\mbox{\scriptsize{max}}}=1$.
As such, TUXV performs just one additional matrix multiply with a reflector block of dimension $n \times k$ over what is performed by TRQRCP.
These results show that TUXV requires only a modest increase in processing time over optimized truncated QR.
Furthermore, TUXV shows significant improvement in approximation quality over both QRCP and RQRCP.
Quality experiments include the truncated SVD and show that TUXV makes significant progress in approaching this theoretical optimum at a tiny fraction of the cost.
A visual quality comparison of each algorithm is also provided by reconstructing low-rank approximations of a test image.

\subsection{Full decomposition time}
The first set of experiments examine scaling of decomposition time versus problem dimensions for several full matrix decompositions.
These experiments were run on a single node of the NERSC machine Edison.
Each node has two 12-core Intel processors.
Our algorithms were written in Fortran90 with OpenMP.
Subroutines were linked with Intel's Math Kernel Library.
Each matrix used is randomly generated GIID.
The same random matrix is submitted to each algorithm.
Order scaling results in which rows $m$ and columns $n$ scaled together are shown in Figure~\ref{fig:fullOrder}.
Remaining experimental results can be found in the Appendix in Figure~\ref{fig:fullRow}, Figure~\ref{fig:fullCol}, and Figure~\ref{fig:fullPara}.

\begin{table}[!htb]
\centering
\footnotesize
\setlength{\tabcolsep}{5pt}
\begin{tabular}{|l|l|}\hline
	\textbf{Subroutine}	& \textbf{Description} \\ \hline\hline
	\textbf{dgeqr2} & LAPACK level-2 BLAS implementation of QR.\\ \hline
	\textbf{dgesvd} & LAPACK singular value decomposition.\\ \hline
	\textbf{dgeqp3}	& LAPACK competing implementation of QRCP.\\ \hline
	\textbf{rsrqrcp}& Alg~\ref{alg:rsrqrcp} Repeated-Sampling Randomized QRCP.\\ \hline
	\textbf{rqrcp}  & Alg~\ref{alg:rqrcp} Randomized QRCP with sample update.\\ \hline
	\textbf{dgeqrf}	& LAPACK level-3 BLAS implementation of QR.\\ \hline
\end{tabular}
\caption{These algorithms are compared in full decomposition scaling experiments.
Rank-revealing subroutines are \textbf{dgesvd}, \textbf{dgeqp3}, \textbf{rsrqrcp}, and \textbf{rqrcp}.
The subroutine \texttt{dgeqrf} demonstrates the performance limit attainable if no pivoting is attempted.
We also include \texttt{dgeqr2} to show the historical evolution of these algorithms.}
\end{table}

\begin{figure}[!htb]
\label{fig:fullOrder}
\centering
\includegraphics[width=0.8\textwidth]{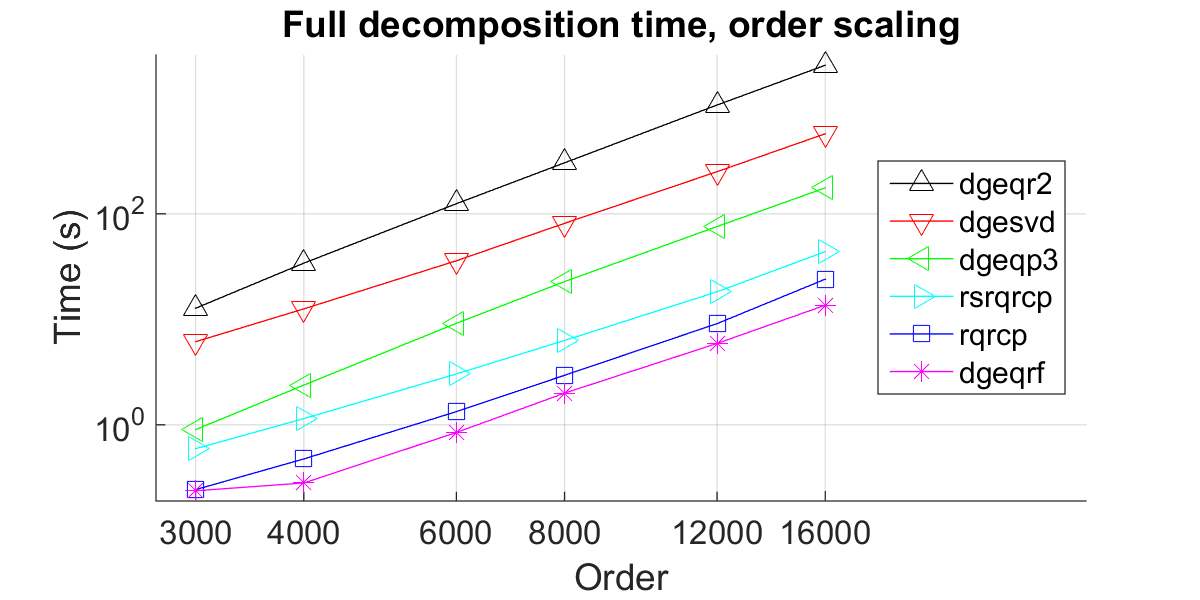}
\caption{24 cores, $m=n$ scaled.}
\end{figure}

\begin{table}[!htb]
\footnotesize
\centering
\setlength{\tabcolsep}{5pt}
\begin{tabular}{|c|c|c|c|c|c|c|}\hline
	\textbf{Order}	&\texttt{dgeqr2}	&\texttt{dgesvd}	&\texttt{dgeqp3}	&\texttt{rsrqrcp}	&\texttt{rqrcp}	&\texttt{dgeqrf}	\\ \hline\hline
	\textbf{ 3k}		&  12.75	&   6.13	&   0.90	&   0.59	&   0.24	&   0.24	\\ \hline
	\textbf{ 4k}		&  34.08	&  12.56	&   2.38	&   1.15	&   0.47	&   0.28	\\ \hline
	\textbf{ 6k}		&  124.7	&  35.90	&   9.16	&   3.04	&   1.33	&   0.85	\\ \hline
	\textbf{ 8k}		&  304.6	&  81.28	&  22.79	&   6.30	&   2.95	&   2.00	\\ \hline
	\textbf{12k}		&   1075	&  250.9	&  75.78	&  18.24	&   9.11	&   5.91	\\ \hline
	\textbf{16k}		&   2563	&  574.5	&  176.2	&  43.84	&  24.03	&  13.61	\\ \hline
\end{tabular}
\caption{Full decomposition order scaling time (s)}
\end{table}

\FloatBarrier

\subsection{Truncated decomposition time} The second set of tests compares truncated approximations using a given truncation rank $k$.
Since the proprietary optimized implementations of LAPACK functions were unavailable for modification, each algorithm was rewritten and adjusted to halt at the desired approximation rank to provide fair comparisons.
Again, each matrix used is GIID and the same matrix is submitted to each algorithm.
Tests for which $m=n=12000$ and $k$ scales are shown in Figure~\ref{fig:truncRank}.
Remaining tests can be found in the Appendix in Figure~\ref{fig:truncRow}, Figure~\ref{fig:truncCol}, and Figure~\ref{fig:truncPara}.

\begin{table}[!htb]
\centering
\footnotesize
\setlength{\tabcolsep}{5pt}
\begin{tabular}{|l|l|}\hline
	\textbf{Subroutine}	& \textbf{Description} \\ \hline\hline
	\textbf{qrcp}    & Alg~\ref{alg:qrcp3} QR with Column Pivoting and blocked trailing update.\\ \hline
	\textbf{rsrqrcp} & Alg~\ref{alg:rsrqrcp} Repeated-Sampling Randomized QRCP.\\ \hline
	\textbf{tuxv}    & Alg~\ref{alg:tuxv} Approximate truncated SVD. No trailing update.\\ \hline
	\textbf{rqrcp}   & Alg~\ref{alg:rqrcp} Randomized QRCP with sample update.\\ \hline
	\textbf{qr}      & Alg~\ref{alg:rqrcp}* QR with blocked trailing update.\\ \hline
	\textbf{trqrcp}  & Alg~\ref{alg:trqrcp} Truncated Randomized QRCP with sample update. No trailing update.\\ \hline
\end{tabular}
\caption{These algorithms are compared in truncated decomposition scaling experiments.
Comparing \textbf{rsrqrcp} with \textbf{rqrcp} reveals the cost of repeated sampling.
Comparing \textbf{rqrcp} with \textbf{qr} reveals the cost of pivot selection from the sample matrix.
Comparing \textbf{rqrcp} with \textbf{trqrcp} further reveals the cost of computing the trailing matrix update.
* Note that the implementation of level-3 BLAS \textbf{qr} is identical to \textbf{rqrcp} after eliminating all sample operations and pivoting.}
\end{table}

\begin{figure}[!htb]
\label{fig:truncRank}
\centering
\includegraphics[width=0.8\textwidth]{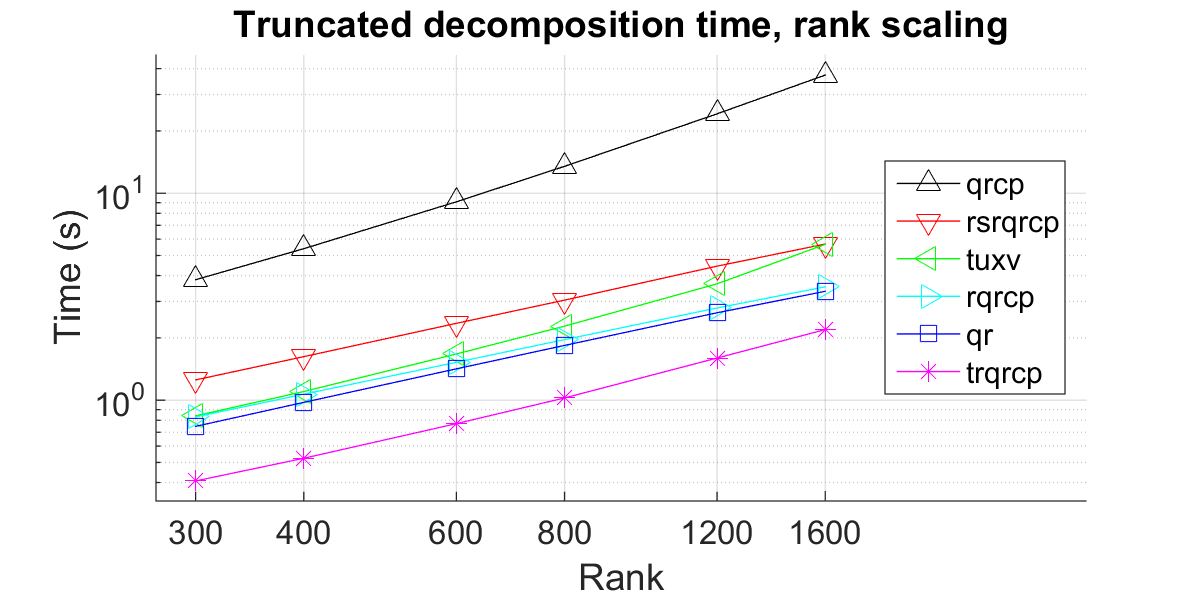}
\caption{24 cores, $m=12000$, $n=12000$, $k$ scaled.}
\end{figure}

\begin{table}[!htb]
\footnotesize
\centering
\setlength{\tabcolsep}{5pt}
\begin{tabular}{|c|c|c|c|c|c|c|}\hline
	\textbf{Rank}	&\texttt{qrcp}	&\texttt{rsrqrcp}	&\texttt{tuxv}	&\texttt{rqrcp}	&\texttt{qr}	&\texttt{trqrcp}	\\ \hline\hline
	\textbf{  300}		&   3.82	&   1.25	&   0.84	&   0.83	&   0.75	&   0.41	\\ \hline
	\textbf{  400}		&   5.40	&   1.62	&   1.10	&   1.07	&   0.98	&   0.52	\\ \hline
	\textbf{  600}		&   9.09	&   2.35	&   1.67	&   1.52	&   1.42	&   0.77	\\ \hline
	\textbf{  800}		&  13.49	&   3.04	&   2.28	&   1.96	&   1.84	&   1.02	\\ \hline
	\textbf{ 1200}		&  24.19	&   4.44	&   3.65	&   2.79	&   2.65	&   1.60	\\ \hline
	\textbf{ 1600}		&  37.25	&   5.68	&   5.66	&   3.53	&   3.36	&   2.19	\\ \hline
\end{tabular}
\caption{Truncated decomposition rank scaling time (s)}
\end{table}

\FloatBarrier

\subsection{Decomposition quality}
\label{subsec:decompquality}

Matrix decomposition quality is compared for the proposed algorithms using three test cases from the San Jose State University Singular Matrix Database.
These test matrices are: \texttt{FIDAP/ex33}, \texttt{HB/lock2232}, and \texttt{LPnetlib/lpi\_gran}.
Each matrix is approximated using a sequence of low-rank decompositions resulting from each algorithm shown.
Relative approximation error in the Frobenius norm is plotted against the corresponding approximation rank.
Plot axes have been chosen to magnify the differences among the algorithms shown.
Relative approximation errors that do not appear on the plot have dropped below the order of machine epsilon, $10\varepsilon \approx 10^{-15}$.
The same test is also shown for a matrix corresponding to a gray-scale image of a differential gear.
Image credit, Alex Kovach~\cite{diffGear}. Test results for \texttt{FIDAP/ex33} and the differential gear are shown in Figure~\ref{fig:ex33} and Figure~\ref{fig:diffGearErr}.
Results for \texttt{HB/lock2232} and \texttt{LPnetlib/lpi\_gran} are found in the appendix: Figure~\ref{fig:lock2232} and Figure~\ref{fig:lpi_gran} respectively.

At the top of each plot we have QR without pivoting.
In order to produce results as competitive as possible QR is applied after presorting columns in order of descending 2-norms.
Despite this modification, QR produces the poorest results.
In several cases the QR approximation error drops much more slowly than all other algorithms.
When this occurs the expanding approximation basis continues to fail to capture the principal components of the trailing matrix.
This demonstrates that QR is not a rank-revealing algorithm.

Below QR we have RQRCP, RSRQRCP, and QRCP.
Recall that RQRCP uses the sample update formula whereas RSRQRCP forms a new sample by multiplying the updated trailing matrix with a new GIID compression matrix after each complete block.
These plots show that both RQRCP and RSRQRCP generally perform as well as QRCP.

Below the QRCP-like algorithms are TUXV and SVD.
The SVD gives the theoretically minimal approximation error for each rank.
In each case, TUXV produces approximation error closer to that of the SVD.

In Figure~\ref{fig:diffGear} we also compare approximation quality by reconstructing the image of the differential gear using low-rank approximations from selected algorithms.
Again, truncated QR shows the poorest reconstruction quality despite presorting.
Both truncated QRCP and TRQRCP produce better results, however, close inspection shows similar fine defects in the reconstructed images.
Reconstruction with RSRQRCP is omitted; Figure~\ref{fig:diffGearErr} shows that the reconstruction would be as good as TRQRCP.
Finally, we have reconstructions using TUXV and the truncated SVD which appear to be indistinguishable from the original.

\begin{figure}[!htb]
	\centering
	\includegraphics[width=0.8\textwidth]{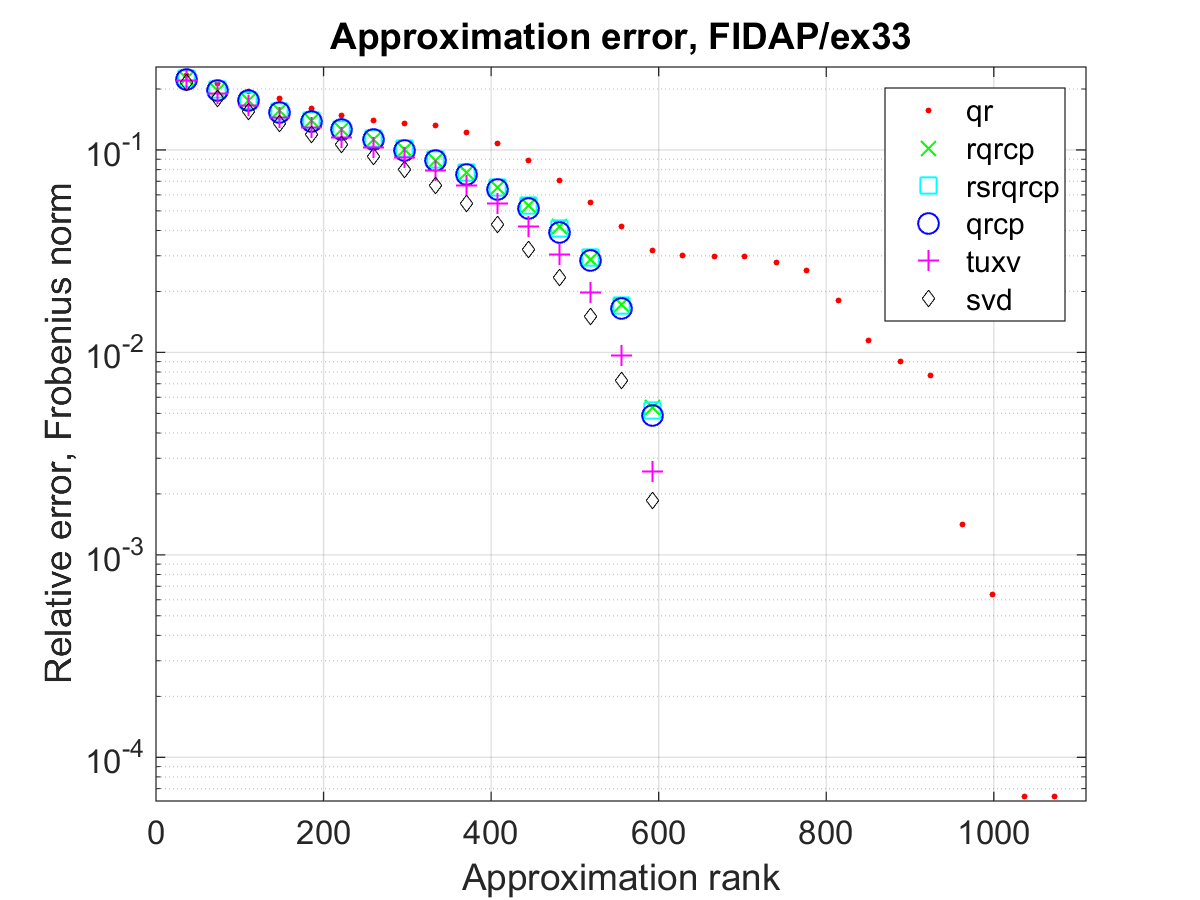}
	\caption{Matrix: FIDAP/ex33. $1733 \times 1733$.}
	\label{fig:ex33}
\end{figure}

\begin{figure}[!htb]
	\centering
	\includegraphics[width=0.8\textwidth]{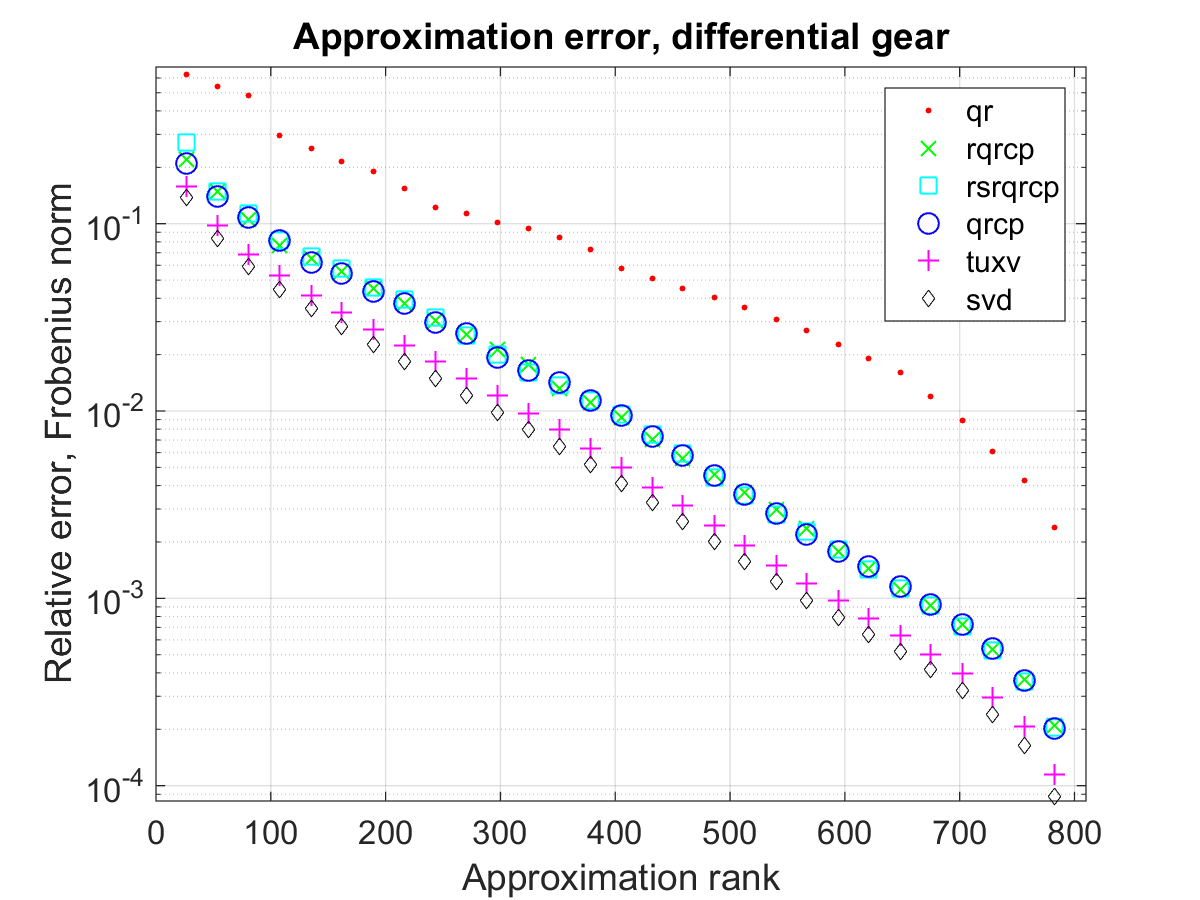}
	\caption{Matrix: Differential Gear~\cite{diffGear}. $1280 \times 804$.}
	\label{fig:diffGearErr}
\end{figure}

\begin{figure}[!htb]
	\centering
	\includegraphics[width=0.9\textwidth]{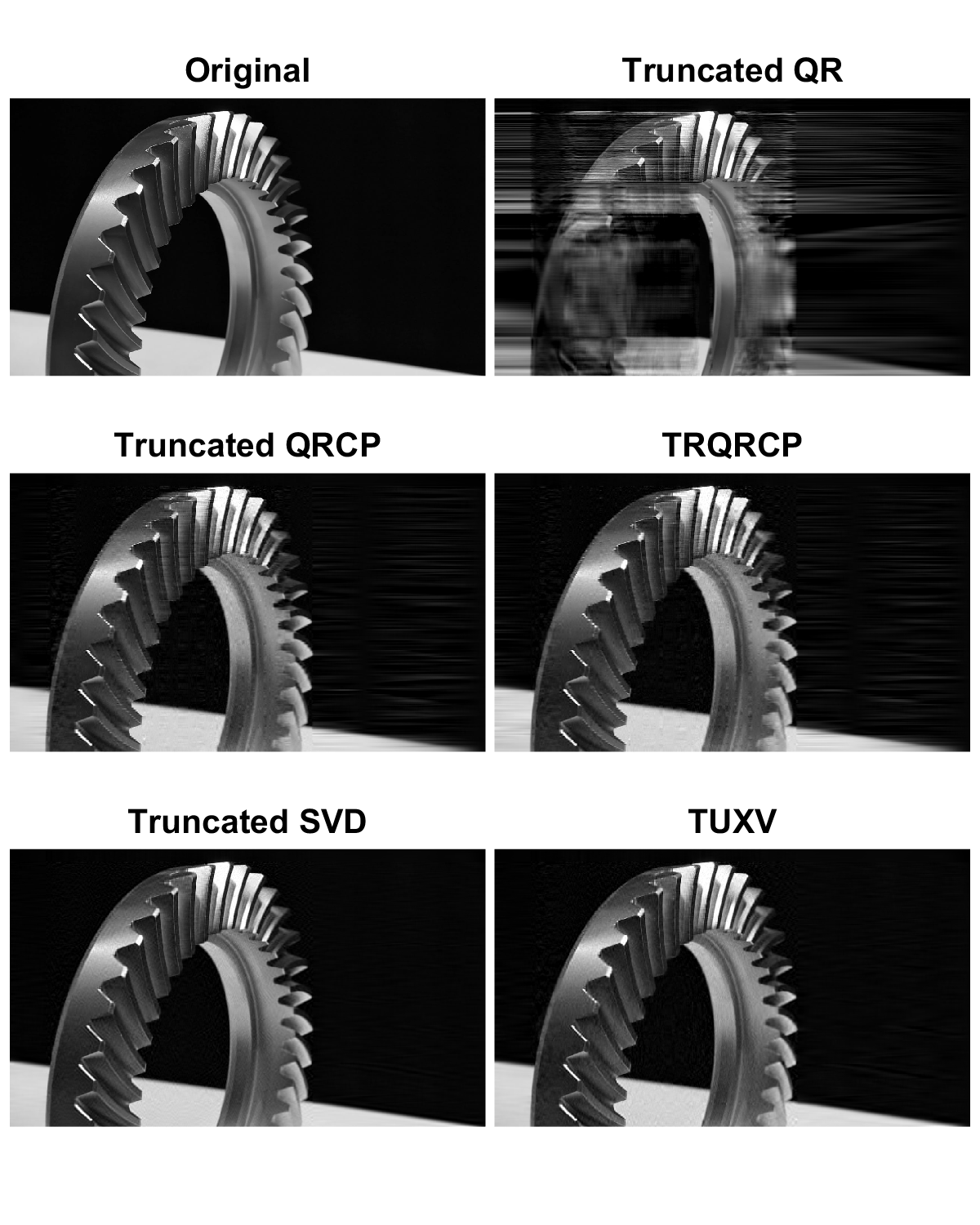}
	\caption{
Low-rank image reconstruction comparison.
Reconstructions are computed from a $1280 \times 804$ grayscale version.
Approximations are reconstructed with rank $k=80$.
Truncated QR is computed with columns presorted by descending 2-norm.
TRQRCP yields nearly the same visual approximation quality as truncated QRCP.
Likewise, TUXV produces an approximation visually similar to the optimal truncated SVD.}
	\label{fig:diffGear}
\end{figure}

\section{Conclusion}

RQRCP achieves pivoting quality near QRCP but at true level-3 BLAS performance and parallel scalability.
This makes RQRCP the algorithm of choice for applications in which a full factorization is required.
For low-rank approximations TRQRCP offers another performance advantage by both halting early and avoiding the trailing update to the original matrix.
For a modest increase in computation time, TUXV further improves the low-rank approximations obtained from TRQRCP and approaches the truncated SVD in quality.
These algorithms open a new performance domain for large matrix factorizations that we believe will be extremely useful in data analysis and possibly machine-learning applications.

Future work on distributed memory implementations would be particularly useful for extremely large matrices that would otherwise be impractical to factorize with QRCP.
Furthermore, existing applications that had to settle for QR due to performance constraints may now find improved stability at very little cost by switching to RQRCP.

\section*{Acknowledgements}

We extend our gratitude to Chris Melgaard, Professor Laura Grigori, and Professor James Demmel for useful conversations regarding this work.
We also thank the reviewers for their hard work and insightful feedback which was very useful in improving the quality of this work.

\nocite{DGGX2013}
\nocite{BDHS2011}
\nocite{B1991}
\nocite{GE1996}
\nocite{Martinsson201147}
\nocite{liberty2007randomized}
\nocite{golub13}
\nocite{journals/siamrev/HalkoMT11}

\bibliographystyle{abbrv}

\bibliography{references}

\begin{thebibliography}{10}

\bibitem{BDHS2011}
G.~Ballard, J.~Demmel, O.~Holtz, and O.~Schwartz.
\newblock Minimizing communication in numerical linear algebra.
\newblock {\em SIAM J. Mat.\ Anal.\ Appl.}, 32(3):866--901, 2011.

\bibitem{B1991}
C.~H. Bischof.
\newblock A parallel {$QR$} factorization algorithm with controlled local
  pivoting.
\newblock {\em SIAM J. Sci.\ Stat.\ Comp.}, 12(1):36--57, Jan. 1991.

\bibitem{BischofVanLoan1987}
C.~H. Bischof and C.~V. Loan.
\newblock The {$WY$} representation for products of householder matrices.
\newblock {\em SIAM J. Sci.\ Stat.\ Comp.}, 8(1):2--13, 1987.

\bibitem{C1987}
T.~F. Chan.
\newblock Rank revealing {$QR$} factorizations.
\newblock {\em Linear Algebra Appl.}, 88--89:67--82, 1987.

\bibitem{CH1992}
T.~F. Chan and P.~C. Hansen.
\newblock Some applications of the rank revealing {$QR$} factorization.
\newblock {\em SIAM J. Sci. Stat. Comput.}, 13(3):727--741, May 1992.

\bibitem{DGGX2013}
J.~Demmel, L.~Grigori, M.~Gu, and H.~Xiang.
\newblock Communication avoiding rank revealing {$QR$} factorization with
  column pivoting.
\newblock Technical Report UCB/EECS-2013-46, EECS Department, University of
  California, Berkeley, May 2013.

\bibitem{journals/siamsc/DemmelGHL12}
J.~Demmel, L.~Grigori, M.~Hoemmen, and J.~Langou.
\newblock Communication-optimal parallel and sequential {$QR$} and {$LU$}
  factorizations.
\newblock {\em SIAM J. Scientific Computing}, 34(1), 2012.

\bibitem{Duersch2014}
J.~A. Duersch and M.~Gu.
\newblock Randomized strong rank-revealing {$QR$} factorization.
\newblock UC Berkeley, Spring 2014, Math 273 final project presentation, May
  2014.

\bibitem{golub13}
G.~H. Golub and C.~F. van Loan.
\newblock {\em Matrix Computations}.
\newblock JHU Press, 4th edition, 2013.

\bibitem{GE1996}
M.~Gu and S.~C. Eisenstat.
\newblock Efficient algorithms for computing a strong rank-revealing {$QR$}
  factorization.
\newblock {\em SIAM J. Sci.\ Comput.}, 17(4):848--869, July 1996.

\bibitem{journals/siamrev/HalkoMT11}
N.~Halko, P.-G. Martinsson, and J.~A. Tropp.
\newblock Finding structure with randomness: Probabilistic algorithms for
  constructing approximate matrix decompositions.
\newblock {\em SIAM Review}, 53(2):217--288, 2011.

\bibitem{journals/na/HuckabyC03}
D.~A. Huckaby and T.~F. Chan.
\newblock On the convergence of {Stewart's} {$QLP$} algorithm for approximating
  the {SVD}.
\newblock {\em Numerical Algorithms}, 32(2-4):287--316, 2003.

\bibitem{JL1984}
W.~B. Johnson and J.~Lindenstrauss.
\newblock Extensions of {Lipschitz} mappings into a {Hilbert} space.
\newblock {\em Contemporary Mathematics}, 26:189--206, 1984.

\bibitem{diffGear}
A.~Kovach.
\newblock Differential gear, 2016.

\bibitem{liberty2007randomized}
E.~Liberty, F.~Woolfe, P.~Martinsson, V.~Rokhlin, and M.~Tygert.
\newblock Randomized algorithms for the low-rank approximation of matrices.
\newblock {\em Proceedings of the National Academy of Sciences}, 104(51):20167,
  2007.

\bibitem{Martinsson2015}
P.~G. Martinsson.
\newblock Blocked rank-revealing {$QR$} factorizations: How randomized sampling
  can be used to avoid single-vector pivoting.
\newblock Preprint on arXiv, May 2015.

\bibitem{Martinsson201147}
P.~G. Martinsson, V.~Rokhlin, and M.~Tygert.
\newblock A randomized algorithm for the decomposition of matrices.
\newblock {\em Applied and Computational Harmonic Analysis}, 30(1):47 -- 68,
  2011.

\bibitem{Quintana-Orti:1998:BVQ:300151.300158}
G.~Quintana-Ort\'{\i}, X.~Sun, and C.~H. Bischof.
\newblock A {BLAS}-3 version of the {$QR$} factorization with column pivoting.
\newblock {\em SIAM J. Sci. Comput.}, 19(5):1486--1494, Sept. 1998.

\bibitem{SchreiberVanLoan1989}
R.~Schreiber and C.~V. Loan.
\newblock A storage-efficient {$WY$} representation for products of householder
  transformations.
\newblock {\em SIAM J. Sci.\ Stat.\ Comp.}, 10(1):53--57, 1989.

\bibitem{journals/siamsc/Stewart99}
G.~W. Stewart.
\newblock The {$QLP$} approximation to the singular value decomposition.
\newblock {\em SIAM J. Scientific Computing}, 20(4):1336--1348, 1999.

\end{thebibliography}

\newpage

\section{Appendix}

\begin{figure}[!htb]
    \centering
    \includegraphics[width=0.8\textwidth]{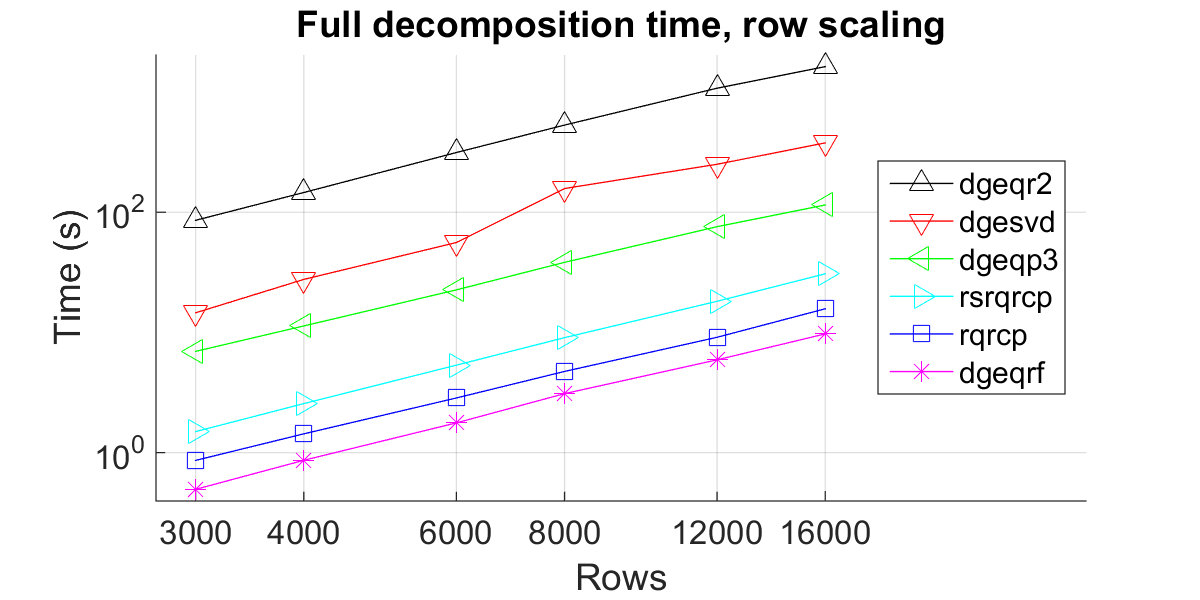}
    \caption{24 cores, $m$ scaled, $n=12000$.}
    \label{fig:fullRow}
\end{figure}

\begin{figure}[!htb]
    \centering
    \includegraphics[width=0.8\textwidth]{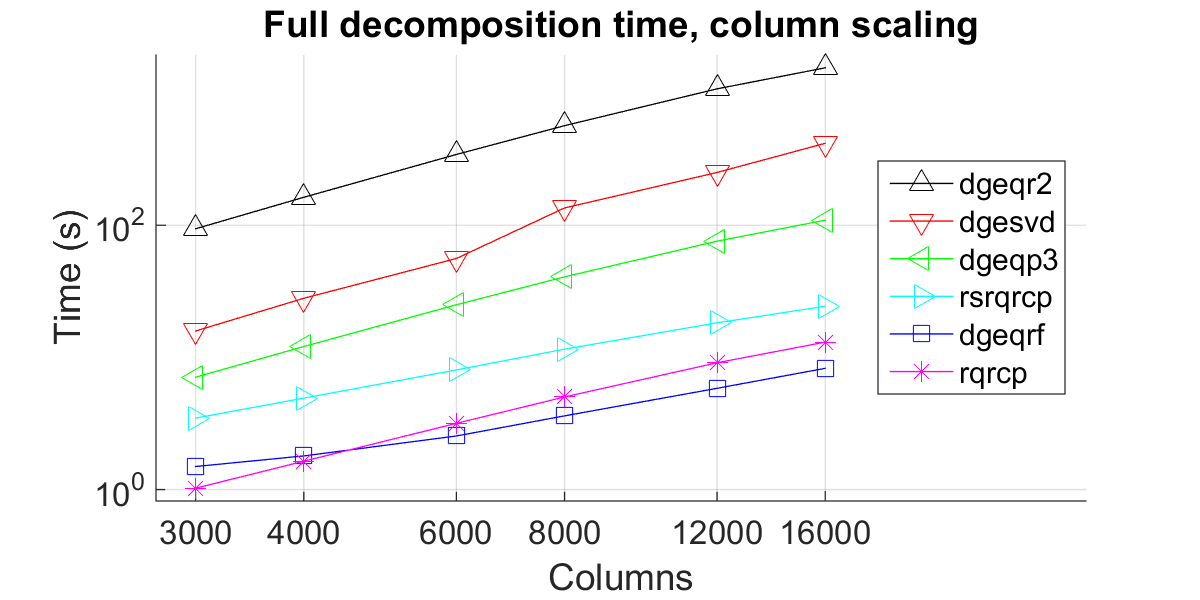}
    \caption{24 cores, $m=12000$, $n$ scaled.}
    \label{fig:fullCol}
\end{figure}

\begin{figure}[!htb]
\centering
\includegraphics[width=0.8\textwidth]{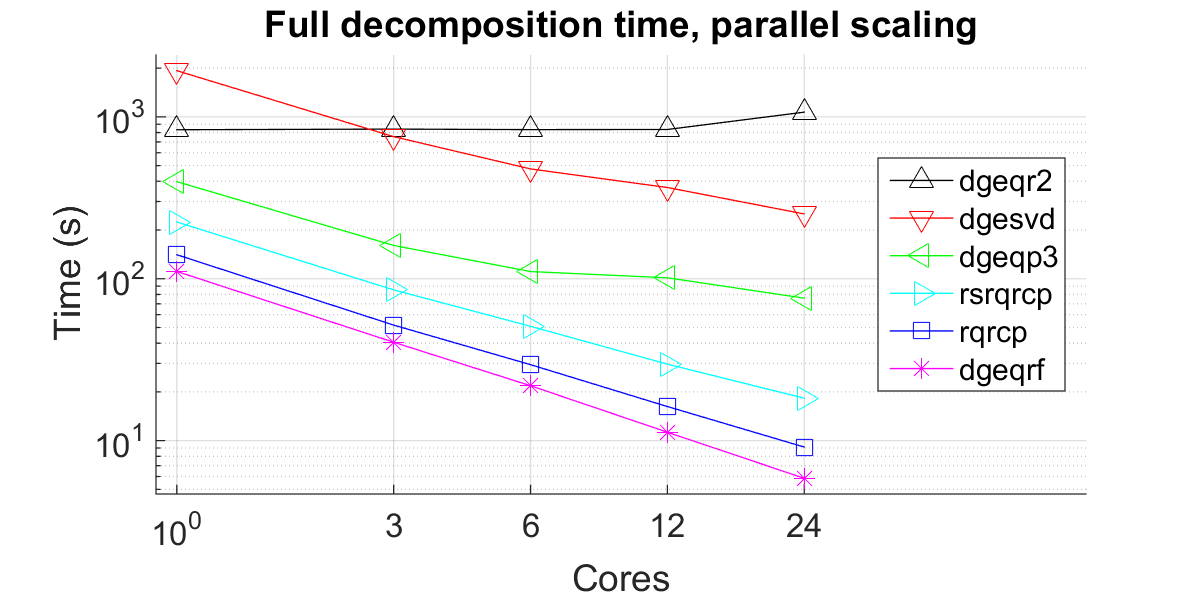}
\caption{Cores scaled, $m=12000$, $n=12000$.}
\label{fig:fullPara}
\end{figure}

\begin{figure}[!htb]
    \centering
    \includegraphics[width=0.8\textwidth]{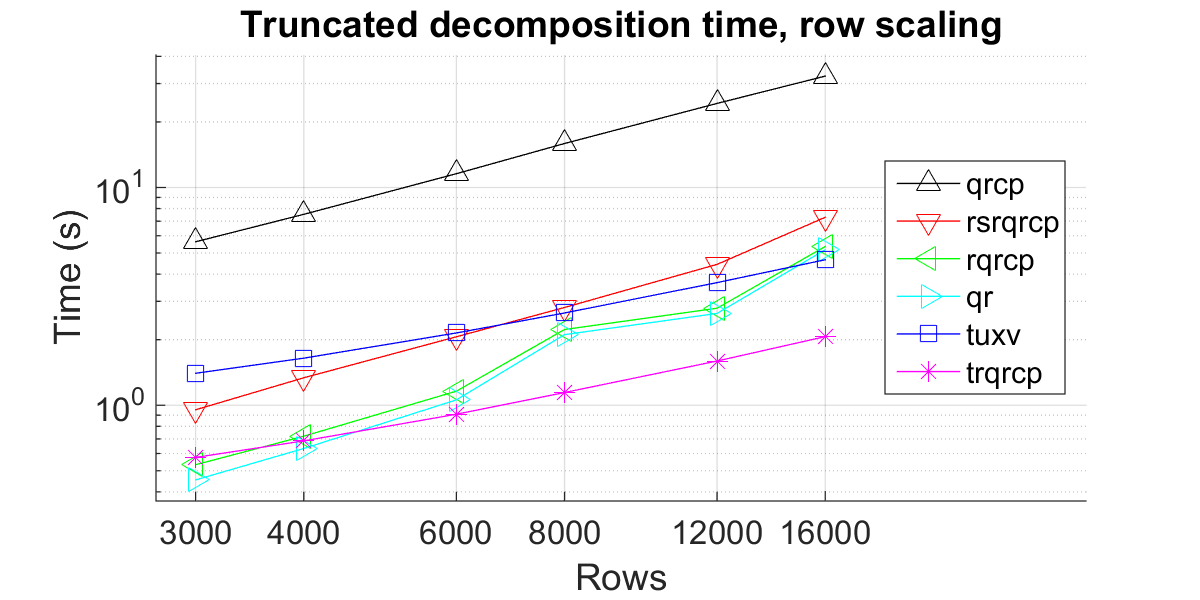}
    \caption{24 cores, $m$ scaled, $n=12000$, $k=1200$.}
\label{fig:truncRow}
\end{figure}

\begin{figure}[!htb]
    \centering
    \includegraphics[width=0.8\textwidth]{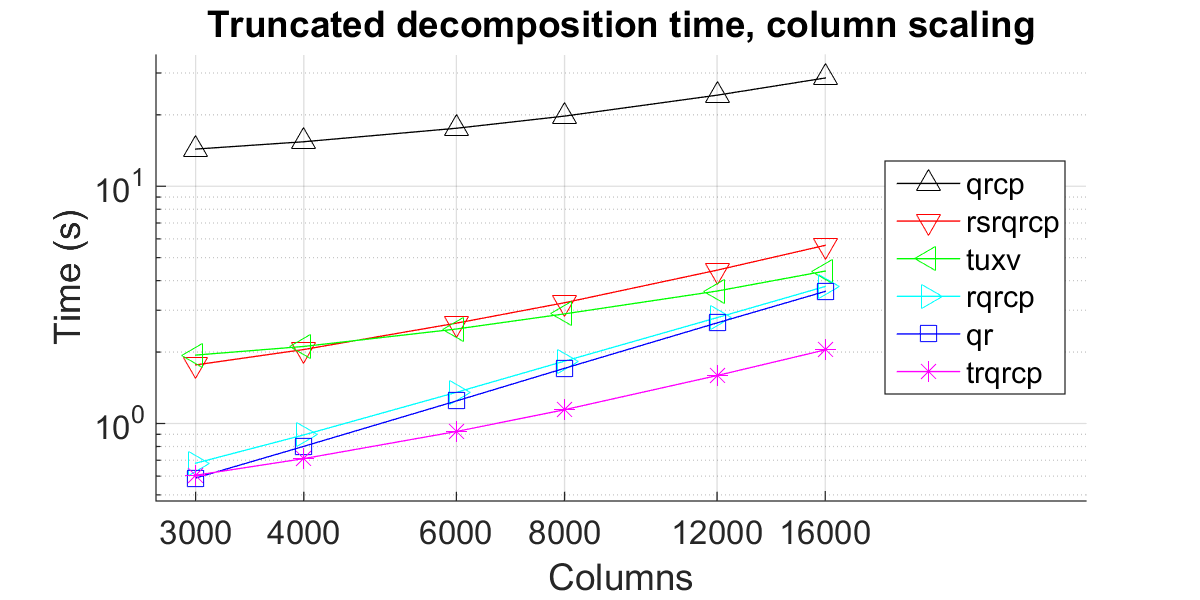}
    \caption{24 cores, $m=12000$, $n$ scaled, $k=1200$.}
    \label{fig:truncCol}
\end{figure}

\begin{figure}[!htb]
\centering
\includegraphics[width=0.8\textwidth]{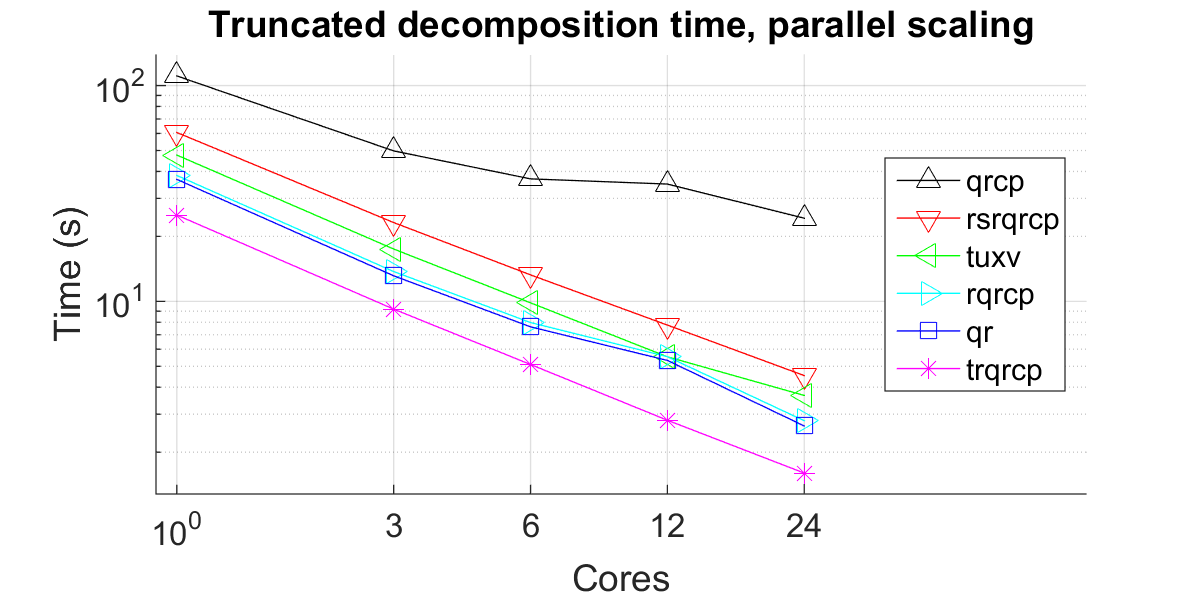}
\caption{Cores scaled, $m=12000$, $n=12000$, $k=1200$.}
\label{fig:truncPara}
\end{figure}

\begin{table}[!htb]
\footnotesize
\centering
\setlength{\tabcolsep}{5pt}
\caption{Full decomposition row scaling time (s)}
\begin{tabular}{|c|c|c|c|c|c|c|}\hline
	\textbf{Rows}	&\texttt{dgeqr2}	&\texttt{dgesvd}	&\texttt{dgeqp3}	&\texttt{rsrqrcp}	&\texttt{rqrcp}	&\texttt{dgeqrf}	\\ \hline\hline
	\textbf{ 3k}		&  85.73	&  14.57	&   6.96	&   1.50	&   0.86	&   0.50	\\ \hline
	\textbf{ 4k}		&  146.0	&  27.64	&  11.33	&   2.56	&   1.43	&   0.86	\\ \hline
	\textbf{ 6k}		&  314.1	&  56.02	&  22.61	&   5.35	&   2.85	&   1.77	\\ \hline
	\textbf{ 8k}		&  531.4	&  157.7	&  38.33	&   9.05	&   4.74	&   3.10	\\ \hline
	\textbf{12k}		&   1077	&  250.8	&  75.85	&  18.17	&   9.12	&   5.92	\\ \hline
	\textbf{16k}		&   1628	&  378.0	&  115.1	&  30.77	&  15.74	&   9.73	\\ \hline
\end{tabular}
\end{table}

\begin{table}[!htb]
\footnotesize
\centering
\setlength{\tabcolsep}{5pt}
\caption{Full decomposition column scaling time (s)}
\begin{tabular}{|c|c|c|c|c|c|c|}\hline
	\textbf{Columns}	&\texttt{dgeqr2}	&\texttt{dgesvd}	&\texttt{dgeqp3}	&\texttt{rsrqrcp}	&\texttt{dgeqrf}	&\texttt{rqrcp}	\\ \hline\hline
	\textbf{ 3k}		&  94.05	&  15.85	&   7.09	&   3.47	&   1.49	&   1.02	\\ \hline
	\textbf{ 4k}		&  162.8	&  28.02	&  12.09	&   4.92	&   1.80	&   1.64	\\ \hline
	\textbf{ 6k}		&  343.2	&  56.03	&  25.03	&   8.04	&   2.54	&   3.17	\\ \hline
	\textbf{ 8k}		&  567.0	&  135.2	&  40.70	&  11.55	&   3.61	&   5.02	\\ \hline
	\textbf{12k}		&   1077	&  250.9	&  75.78	&  18.27	&   5.84	&   9.11	\\ \hline
	\textbf{16k}		&   1554	&  417.0	&  109.2	&  24.41	&   8.27	&  13.05	\\ \hline
\end{tabular}
\end{table}

\begin{table}[!htb]
\footnotesize
\centering
\setlength{\tabcolsep}{5pt}
\caption{Full decomposition parallel scaling time (s)}
\begin{tabular}{|c|c|c|c|c|c|c|}\hline
	\textbf{Cores}	&\texttt{dgeqr2}	&\texttt{dgesvd}	&\texttt{dgeqp3}	&\texttt{rsrqrcp}	&\texttt{rqrcp}	&\texttt{dgeqrf}	\\ \hline\hline
	\textbf{    1}		&  832.8	&   1929	&  397.8	&  224.1	&  140.7	&  111.0	\\ \hline
	\textbf{    3}		&  840.3	&  754.8	&  160.6	&  85.33	&  51.71	&  40.52	\\ \hline
	\textbf{    6}		&  834.2	&  476.9	&  110.8	&  50.83	&  29.45	&  21.83	\\ \hline
	\textbf{   12}		&  835.6	&  366.0	&  101.3	&  29.68	&  16.24	&  11.27	\\ \hline
	\textbf{   24}		&   1069	&  251.4	&  75.89	&  18.28	&   9.12	&   5.85	\\ \hline
\end{tabular}
\end{table}

\begin{table}[!htb]
\footnotesize
\centering
\setlength{\tabcolsep}{5pt}
\caption{Truncated decomposition row scaling time (s)}
\begin{tabular}{|c|c|c|c|c|c|c|}\hline
	\textbf{Rows}	&\texttt{qrcp}	&\texttt{rsrqrcp}	&\texttt{rqrcp}	&\texttt{qr}	&\texttt{tuxv}	&\texttt{trqrcp}	\\ \hline\hline
	\textbf{ 3k}		&   5.63	&   0.95	&   0.53	&   0.45	&   1.40	&   0.58	\\ \hline
	\textbf{ 4k}		&   7.53	&   1.34	&   0.72	&   0.63	&   1.65	&   0.69	\\ \hline
	\textbf{ 6k}		&  11.57	&   2.06	&   1.16	&   1.06	&   2.15	&   0.91	\\ \hline
	\textbf{ 8k}		&  15.92	&   2.81	&   2.23	&   2.11	&   2.65	&   1.14	\\ \hline
	\textbf{12k}		&  24.32	&   4.43	&   2.79	&   2.64	&   3.66	&   1.60	\\ \hline
	\textbf{16k}		&  32.47	&   7.30	&   5.36	&   5.18	&   4.66	&   2.07	\\ \hline
\end{tabular}
\end{table}

\begin{table}[!htb]
\footnotesize
\centering
\setlength{\tabcolsep}{5pt}
\caption{Truncated decomposition column scaling time (s)}
\begin{tabular}{|c|c|c|c|c|c|c|}\hline
	\textbf{Columns}	&\texttt{qrcp}	&\texttt{rsrqrcp}	&\texttt{tuxv}	&\texttt{rqrcp}	&\texttt{qr}	&\texttt{trqrcp}	\\ \hline\hline
	\textbf{ 3k}		&  14.34	&   1.76	&   1.94	&   0.68	&   0.59	&   0.60	\\ \hline
	\textbf{ 4k}		&  15.40	&   2.05	&   2.11	&   0.90	&   0.80	&   0.71	\\ \hline
	\textbf{ 6k}		&  17.55	&   2.65	&   2.50	&   1.35	&   1.24	&   0.93	\\ \hline
	\textbf{ 8k}		&  19.75	&   3.23	&   2.90	&   1.83	&   1.71	&   1.14	\\ \hline
	\textbf{12k}		&  24.20	&   4.43	&   3.62	&   2.79	&   2.65	&   1.59	\\ \hline
	\textbf{16k}		&  28.59	&   5.63	&   4.39	&   3.78	&   3.60	&   2.04	\\ \hline
\end{tabular}
\end{table}

\begin{table}[!htb]
\footnotesize
\centering
\setlength{\tabcolsep}{5pt}
\caption{Truncated decomposition parallel scaling time (s)}
\begin{tabular}{|c|c|c|c|c|c|c|}\hline
	\textbf{Cores}	&\texttt{qrcp}	&\texttt{rsrqrcp}	&\texttt{tuxv}	&\texttt{rqrcp}	&\texttt{qr}	&\texttt{trqrcp}	\\ \hline\hline
	\textbf{    1}		&  110.8	&  60.60	&  47.63	&  38.31	&  36.75	&  25.12	\\ \hline
	\textbf{    3}		&  49.79	&  23.18	&  17.47	&  13.74	&  13.13	&   9.20	\\ \hline
	\textbf{    6}		&  36.94	&  13.26	&   9.84	&   7.97	&   7.62	&   5.10	\\ \hline
	\textbf{   12}		&  34.95	&   7.76	&   5.53	&   5.54	&   5.32	&   2.81	\\ \hline
	\textbf{   24}		&  24.28	&   4.53	&   3.66	&   2.80	&   2.65	&   1.60	\\ \hline
\end{tabular}
\end{table}

\begin{figure}[!htb]
	\centering
	\includegraphics[width=0.9\textwidth]{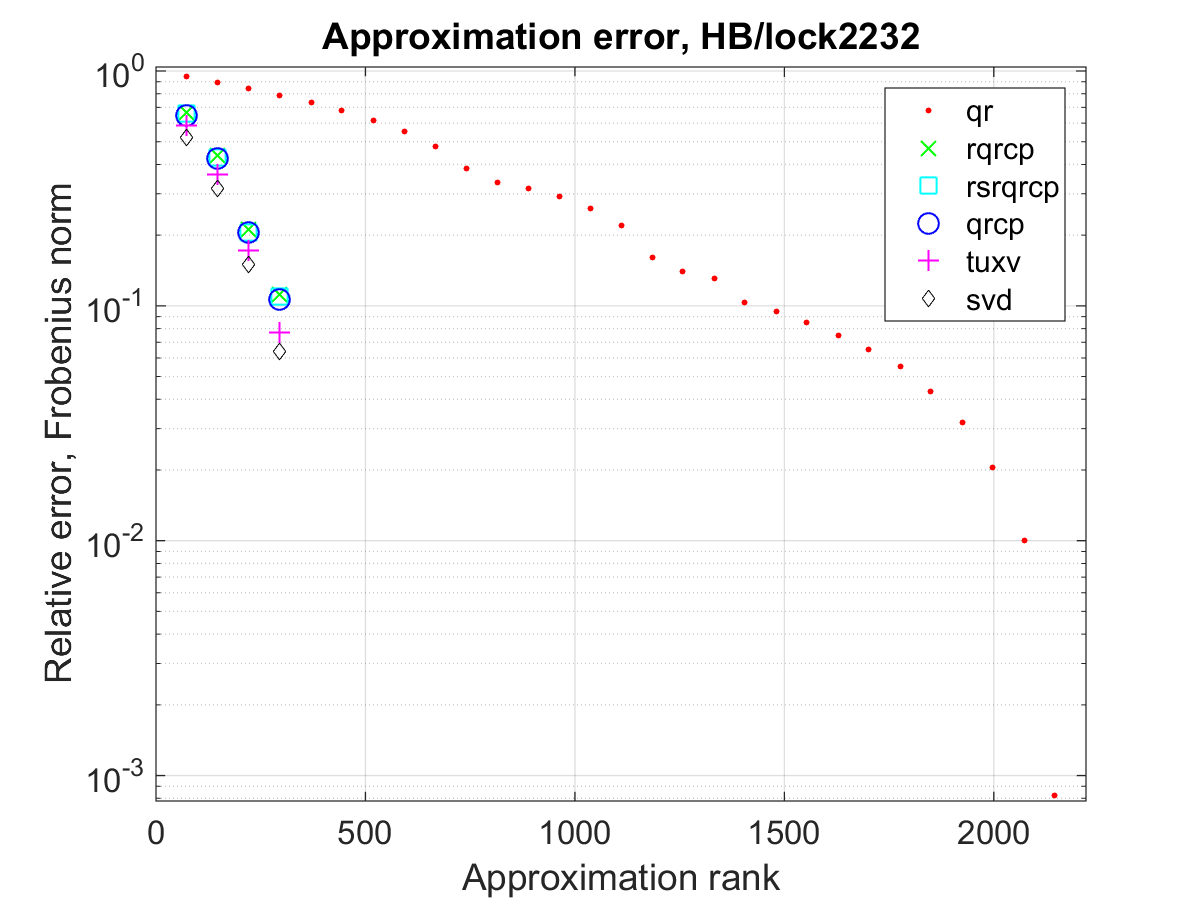}
	\caption{Matrix: HB/lock2232. $2232 \times 2232$.}
	\label{fig:lock2232}
\end{figure}

\begin{figure}[!htb]
	\centering
	\includegraphics[width=0.9\textwidth]{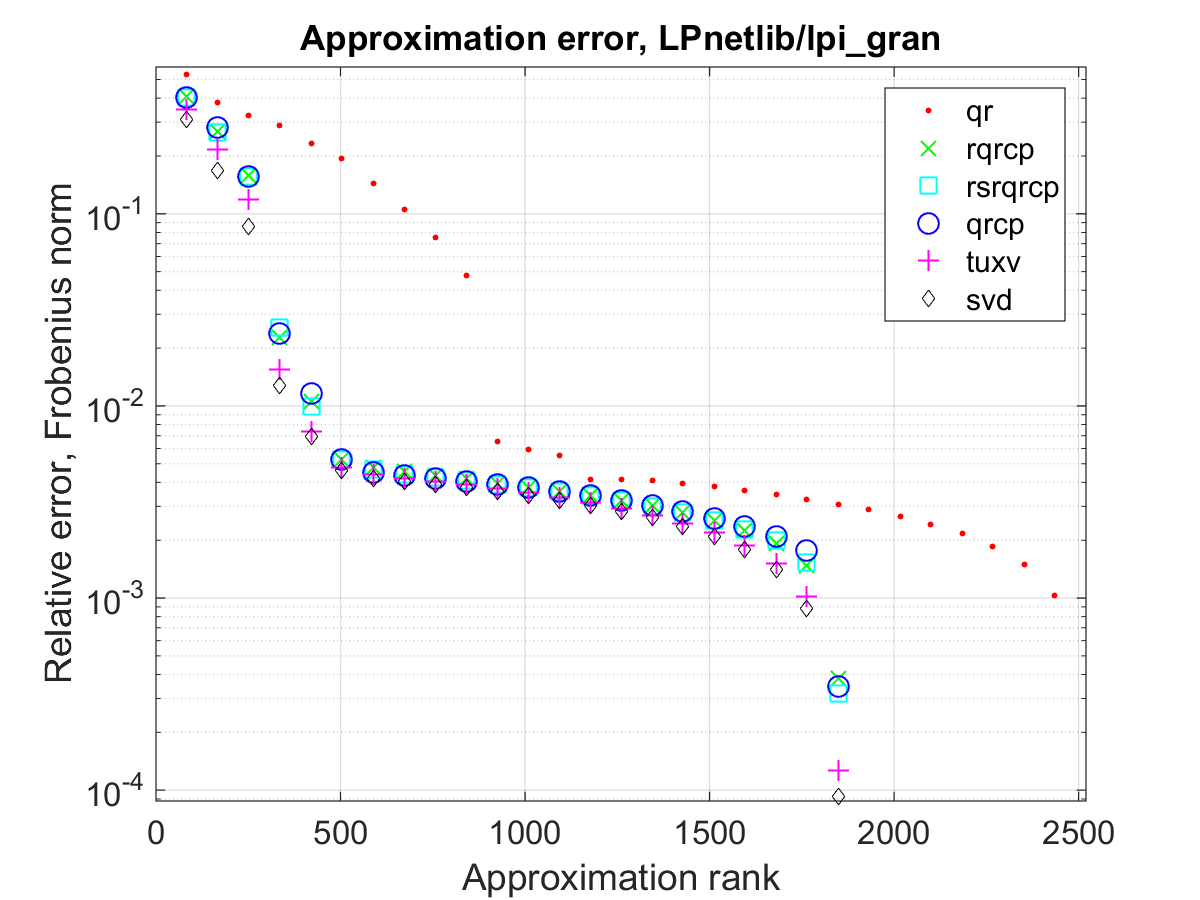}
	\caption{Matrix: LPnetlib/lpi\_gran. $2658 \times 2525$.}
	\label{fig:lpi_gran}
\end{figure}

\end{document}